\documentclass[a4paper,11pt,reqno]{amsart}
\usepackage{graphicx}
\usepackage{mathtools}
\usepackage[english]{babel}

\usepackage{amssymb,enumerate,amsmath}
\usepackage[bookmarks,colorlinks]{hyperref}
\usepackage{tikz}
\usepackage{array}
\usepackage{mathrsfs}
\usepackage{faktor}

\usepackage[left=2cm,right=2cm,top=2cm,bottom=2cm]{geometry}

\newcommand{\NN}{\mathbb{N}}
\newcommand{\ZZ}{\mathbb{Z}}

\newcommand{\PP}{\mathbb{P}}
\newcommand{\OO}{\mathcal{O}}
\newcommand{\B}{\mathcal{B}}
\newcommand{\cN}{\mathcal{N}}
\newcommand{\Supp}{\textnormal{Supp}}

\renewcommand{\AA}{\mathbb{A}}
\newcommand{\GrassFunctor}[2]{\underline{\mathbf{Gr}}_{#1}^{#2}}
\newcommand{\GrassScheme}[2]{\mathbf{Gr}_{#1}^{#2}}
\newcommand{\HilbScheme}[2]{\mathbf{Hilb}^{#2}_{#1}}
\newcommand{\HilbFunctor}[2]{\underline{\mathbf{Hilb}}^{#2}_{#1}}
\newcommand{\GFunctor}[1]{\underline{\mathbf{G}}_{#1}}

\newcommand{\HFunctor}[1]{\underline{\mathbf{H}}_{#1}}
\newcommand{\HScheme}[1]{\mathbf{H}_{#1}}
\newcommand{\MFFunctor}[1]{\underline{\mathbf{Mf}}_{#1}}
\newcommand{\MFScheme}[1]{\mathbf{Mf}_{#1}}
\newcommand{\GSFunctor}[2]{\underline{\mathbf{St}}_{#1}^{#2}}
\newcommand{\GSScheme}[2]{\mathbf{St}_{#1}^{#2}}
\newcommand{\IN}[1]{\textnormal{in}_{#1}}

\newcommand{\Lex}{\mathtt{Lex}}
\newcommand{\EK}{\textsc{ek}}

\newcommand{\Ht}{\mathrm{Ht}}

\newcommand{\T}{\mathrm{T}}

\newcommand{\In}{\textnormal{in}}

\newcommand{\rank}{\textnormal{rk}\,}

\newcommand{\Proj}{\textnormal{Proj}\,}
\newcommand{\Spec}{\textnormal{Spec}\,}

\newcommand{\Sets}{\underline{\textnormal{Sets}}}

\newcommand{\xx}{\mathbf{x}}
\newcommand{\CC}{\mathbf{C}}

\renewcommand{\i}{\mathrm{i}}

\renewcommand{\geq}{\geqslant}
\renewcommand{\leq}{\leqslant}

\numberwithin{equation}{section}
\newtheorem{theorem}{Theorem}[section]
\newtheorem{corollary}[theorem]{Corollary}
\newtheorem{proposition}[theorem]{Proposition}
\newtheorem{lemma}[theorem]{Lemma}
\newtheorem{definition}[theorem]{Definition}

\theoremstyle{definition}
\newtheorem{example}[theorem]{Example}
\newtheorem{remark}[theorem]{Remark}
\newtheorem{notation}[theorem]{Notation}

\def\ecr{\color{black}}

\begin{document}
\title{On the functoriality of marked families}


\author[P.~Lella]{Paolo Lella}
\address{Paolo Lella\\ Dipartimento di Matematica\\ 
         Via Sommarive 14\\ 38123 Povo Trento\\ Italy.}
\email{\href{mailto:paolo.lella@unitn.it}{paolo.lella@unitn.it}}
\urladdr{\url{http://www.paololella.it/}}

\author[M.~Roggero]{Margherita Roggero}
\address{Margherita Roggero\\ Dipartimento di Matematica\\ 
         Via Carlo Alberto 10\\     10123 Torino\\ Italy.}
\email{\href{mailto:margherita.roggero@unito.it}{margherita.roggero@unito.it}}

\thanks{The first author was partially supported by GNSAGA of INdAM, by PRIN 2010--2011 \lq\lq Geometria delle variet\`a algebriche\rq\rq, and by FIRB 2012 \lq\lq Moduli spaces and Applications\rq\rq. The second author was supported by PRIN 2010--2011 \lq\lq Geometria delle variet\`a algebriche\rq\rq.}


\begin{abstract}
The application of methods of computational algebra has recently introduced new tools for the study of Hilbert schemes. The key idea is to define flat families 
of ideals endowed with a scheme structure whose defining equations can be determined by algorithmic procedures. 
For this reason, several
authors developed new methods, based on the combinatorial properties of Borel-fixed ideals, that allow to associate to each  ideal $J$ of this type a scheme $\MFScheme{J}$, called $J$-marked scheme.
In this paper we provide a solid functorial
foundation to marked schemes and show  that the algorithmic procedures introduced in previous papers do not depend on the ring of coefficients.\ecr\ 
 We prove that  for all strongly stable ideals $J$,  the marked schemes 
 $\MFScheme{J}$  can be   embedded in a Hilbert scheme  as  locally closed subschemes, and  that they are open under suitable conditions on  $J$. 
Finally, we generalize Lederer's result about Gr\"obner strata of zero-dimensional ideals, proving that Gr\"obner strata of any ideals are locally closed subschemes of Hilbert schemes.
\end{abstract} 


\maketitle



\section*{Introduction}

This article aims  to give a solid functorial foundation to the theory of  marked schemes  over a strongly stable ideal $J$ introduced  in \cite{CioffiRoggero,BCLR}.
We   describe them in terms of representable functors and     prove that these functors are represented by the schemes constructed in the aforementioned papers.  
Moreover, under mild additional  hypotheses on $J$, these functors turn out to be  subfunctors of a Hilbert functor.  Equations defining the marked schemes
can be effectively computed, therefore these methods allow for effective computations on the  Hilbert schemes. 
  In particular, if we only consider algebras and schemes 
 over a field 
 of characteristic zero, marked schemes $\MFScheme{J}$ with $J$ strongly stable  provide, up to the action of the linear group, an
 open cover of the Hilbert scheme.

For a given monomial ideal $J $ in a polynomial ring $A[x_0, \dots, x_n]$, we consider the collection of all the ideals $I$ such that 
$A[x_0, \dots, x_n] =I\oplus \langle\cN(J)\rangle$, where $\cN(J)$ denotes  the set of monomials not contained in $J$. In the case where $A$ is a field and $J$ strongly stable, this collection
 appears  for the first time in \cite{CioffiRoggero},  where it is called \emph{$J$-marked family}, and it is proved that it can be endowed with a structure of scheme (called $J$-marked scheme) \cite{CioffiRoggero,BCLR}.

    All the ideals $I$ of this collection share the same basis $\cN(J)$
of the quotient algebra $A[x_0, \dots, x_n]/I$, therefore they define subschemes in $\Proj A[x_0, \dots, x_n]$ with the same Hilbert polynomial.  These same  properties hold 
for  Gr\"obner strata, which are schemes parametrizing homogeneous ideals having a fixed monomial ideal as their initial ideal w.r.t.~a given term ordering.
However, we emphasize that marked schemes and Gr\"obner strata are not the same objects. Indeed, in general a $J$-marked scheme strictly contains the Gr\"obner
stratum with initial ideal $J$ w.r.t.~a fixed term ordering (or even the union of all Gr\"obner strata with initial ideal $J$).
  


The use of Gr\"obner strata in the study of Hilbert schemes is very natural and has
  been discussed since \cite{Bayer82,CarraFerro}. Indeed, the ideals of a Gr\"obner stratum define points on the same Hilbert scheme and 
Gr\"obner strata cover set-theoretically the Hilbert scheme. Thus, several authors addressed the question whether a Gr\"obner stratum can be equipped by a scheme structure and, 
if so, how this scheme is embedded in the Hilbert scheme.

 Notari and Spreafico \cite{NotariSpreafico} prove that every Gr\"obner stratum (considering the reverse lexicographic order) is a locally closed subscheme of the support
 of the Hilbert scheme. Lederer \cite{Lederer} obtains a stronger result in the case of Hilbert schemes of points; working in the affine framework, he proves that Gr\"obner
 strata are locally closed subschemes of the Hilbert scheme. In \cite{LR}, the  authors of the present paper 
 find suitable conditions on the monomial ideal $J$ and on the term ordering that are sufficient to ensure that the Gr\"obner stratum is an open subscheme of
 the Hilbert scheme. 
 
Nevertheless, Gr\"obner strata are in general not sufficient to obtain an open cover of the Hilbert scheme (see \cite{LR,CioffiRoggero}), while we can obtain such an open cover using marked schemes and exploiting the action of the general linear group on the Hilbert scheme. Furthermore,  equations defining a $J$-marked scheme can be computed by some algorithmic procedures developed in \cite{CioffiRoggero,BCLR,BLR}. The key point is a procedure of polynomial reduction, similar to the one for Gr\"obner bases, but that does not need a term ordering (see Definition \ref{riduzione}). 

In this paper, we prove that the procedure of reduction is also \lq\lq natural\rq\rq. Indeed, the reduction works independently of the ring $A$ of coefficients of
the polynomial ring, so that the schemes introduced in \cite{CioffiRoggero,BCLR,BLR} correctly describe the scheme structure of the Hilbert scheme (Theorem \ref{th:markedFunctorsInclusion} and Corollary \ref{cor:mfEmbHilb}).

%
 
In the classical construction of the Hilbert scheme, every point is associated to the homogeneous piece of (a sufficiently large) degree $r$ of the ideal defining the corresponding scheme. At first sight, one could be tempted to consider marked scheme over ideals truncated in the same (large) degree. However, explicit computations of these marked schemes turn out to be in general out of reach,  due to the huge number of variables required. Since the number of variables depends on the degree of the truncation, we develop the theory of marked functors in a wider generality,  considering marked functors over ideals truncated in any degree. In this way we can find marked schemes that correctly describe the local structure of the Hilbert scheme, but that are far easier to compute (Theorem \ref{th:markedFunctorsInclusion}).

Finally, we discuss the relation between marked schemes and Gr\"obner strata, also introducing a representable functor whose representing scheme is in fact a Gr\"obner stratum. For constant Hilbert polynomials, the Gr\"obner strata we define in the projective case coincide with those introduced by Lederer in the affine case. In this paper, we generalize Lederer result to Hilbert polynomials of any degree, proving that Gr\"obner strata are closed subschemes of marked schemes, and so locally closed subschemes of the Hilbert scheme (Theorem \ref{th:gsClosedSubfunctor}).

In the last section, we discuss in details the case of a strongly stable ideal defining a zero-dimensional subscheme of $\PP^3$ of degree $7$. We compute explicitly the equations of marked schemes and Gr\"obner strata of different truncations, showing the relations among them and exhibiting several phenomena described in the paper.

\section{Marked bases}
In this section, we recall the main definitions concerning sets of polynomials marked over a monomial ideal $J$ and we describe some properties of an ideal generated
by such a set, assuming that $J$ is strongly stable. First, let us fix some notation. Throughout the paper, we 
will consider noetherian rings. We will denote by $\ZZ[\xx]$ the polynomial ring $\ZZ[x_0,\ldots,x_n]$ and by $\PP^n_\ZZ$ the projective space $\Proj \ZZ[\xx]$.
For any ring $A$, $A[\xx]$ will denote the polynomial ring $A \otimes_\ZZ \ZZ[\xx]$ in $n+1$ variables with coefficients in $A$ and $\PP_A^n$ will be the scheme
$\Proj A[\xx] = \PP^n_\ZZ \times_{\Spec \ZZ} \Spec A$. For every integer $s$, we denote by $A[\xx]_s$ the graded component  of degree $s$, and we set
$D_s:=D\cap  A[\xx]_s$ for every  $D\subseteq A[\xx]$ .

We denote monomials in multi-index notation. For any element $\alpha = (\alpha_0,\ldots,\alpha_n) \in \NN^{n+1}$, $x^\alpha$ will be the monomial
$x_0^{\alpha_0}\cdots x_n^{\alpha_n}$ and $\vert\alpha\vert$ will be its degree. Given a set of homogeneous polynomials $H$ in $A[\xx]$, for emphasizing the dependence 
on the coefficient ring A, we write ${}_A\langle H\rangle$ for the $A$-module generated by $H$ and ${}_A(H)$ for the ideal in $A[\xx]$ generated by $H$. We will omit this
subscript when no ambiguity can arise, for instance when only one ring $A$ is involved.

If  $J$ is a monomial  ideal in $A[\xx]$, then $\B_J$ is its minimal set of generators and  $\cN(J)$ is the set of monomials not contained in $J$.

\begin{remark}
A monomial ideal is determined by the set of monomials it contains. In the following, by abuse of notation,  we will use the same letter to denote all  monomial ideals having 
the same set of monomials, even in polynomial rings with different rings of coefficients. More formally, if $J$ is a monomial ideal in $\ZZ[\xx]$, we will denote by the same 
symbol $J$ also all the ideals $J\otimes_\ZZ A$. 
\end{remark}

Throughout the paper, we assume the variables ordered as $x_0 < \cdots < x_n$. For any monomial $x^\alpha$, we denote by $\min x^\alpha$ the smallest variable (or equivalently 
its index) dividing $x^\alpha$ and by $\max x^\alpha$ the greatest variable (or its index) dividing the monomial. 
\begin{definition}
An ideal $J \subseteq A[\xx]$ is said \emph{strongly stable} if
\begin{enumerate}[(i)]
\item $J$ is a monomial ideal;
\item if $x^\alpha \in J$, then $\frac{x_i}{x_j} x^\alpha \in J$, for all $x_j \mid x^\alpha$ and $x_i > x_j$.
\end{enumerate}
\end{definition}
These ideals are extensively studied in commutative algebra and widely used in algebraic geometry since they are  related to the \emph{Borel-fixed ideals} \cite{Green}.
Indeed, every strongly stable ideal is Borel-fixed, whereas in general a Borel-fixed ideal does not need to be 
strongly stable. The two notions coincide in polynomial rings with coefficients in a field of characteristic zero.
  Borel-fixed ideals  are involved in some of the most important general results on Hilbert schemes, as for instance the proof of its connectedness given by Hartshorne \cite{HartshorneThesis}. 

Combinatorial properties of strongly stable ideals have been successfully used for designing algorithms inspired by the theory of Gr\"obner bases but not requiring a term
ordering. The role of the term ordering, a total ordering on the set of monomials, is played by a partial order called the \emph{Borel ordering}, given as
the transitive closure of the relation
\[
x^\alpha >_B x^\beta \quad \Longleftrightarrow \quad x_i x^\alpha = x_j x^\beta\text{ and } x_i < x_j.
\]
Moving from this order, it is possible to define reduction procedures which turn out to be noetherian. A detailed description of these techniques are contained in the
papers \cite{CioffiRoggero,BCLR,BLR}. We will now recall some of the main properties needed in the next section.

\begin{definition}
For a polynomial $f\in A[\xx]$, its \emph{support}, denoted by $\Supp(f)$, is the set of monomials appearing in $f$ with non-zero coefficient. We refer to the set of non-zero coefficients of $f$ as \emph{$\xx$-coefficients} of $f$. A \emph{monic marked polynomial} is a polynomial $f \in A[\xx]$ with a specified monomial $\Ht(f)$ of its support,  with coefficient $1_A$.  We  call $\Ht(f)$  the \emph{head term} of $f$ and    we call  $\T(f) := \Ht(f) - f$ the \emph{tail} of $f$  (so that $f = \Ht(f) - \T(f)$). Throughout the paper we describe marked polynomials adding as subscript the multi-index corresponding to the head term, i.e.~we write $f_\alpha$ meaning that $\Ht(f_\alpha) = x^\alpha$. 
\end{definition}
 
\begin{definition}
Let $J \subseteq A[\xx]$ be a strongly stable ideal and let $\B_J$ be the minimal set of generators of $J$. We call \emph{$J$-marked set} a set of monic marked polynomials
\begin{equation*}
  \left\{f_{\alpha} := x^\alpha - \sum_{\mathclap{\cramped{x^\beta \in \cN(J)_{\vert\alpha\vert}}}} c_{\alpha\beta}\, 
  x^\beta \in A[\xx]\ \Bigg\vert\ x^{\alpha} \in \B_J\right\},
  \end{equation*}
where $\Ht(f_{\alpha}) = x^\alpha$ and $c_{\alpha\beta}\in A$. A $J$-marked set $F_J$ is called a \emph{$J$-marked basis} if $A[\xx] = {}_A(F_J) \oplus {}_A\langle \cN(J)\rangle$, i.e.~the monomials of $\cN(J)$ freely generate $A[\xx]/{}_A(F_J)$. 
\end{definition}

  We emphasize that the assumption of the head term to be monic is significant only if the ring of coefficients
 $A$  is not a field. Indeed, if $A$ is a field
 (as done in \cite{CioffiRoggero,BCLR}), a set of marked polynomials can always be modified in a set of monic marked polynomials. 

\smallskip

If $(F_J)$ is a $J$-marked basis, then the scheme $\Proj A[\xx]/(F_J)$ is $A$-flat because the $A$-module $A[\xx]/{}_A(F_J)$ is free. The ideal ${}_A(F_J)$
generated by a $J$-marked basis $F_J$ has the same Hilbert polynomial as the monomial ideal $J$, so that $J$ and ${}_A(F_J)$ define schemes corresponding to closed 
points of the same Hilbert scheme.
Therefore, it is interesting to find theoretical conditions and  effective procedures in order to state  whether a marked set is a marked basis. 

\begin{proposition}[{\cite[Lemma 1.1]{EliahouKervaire}, \cite[Lemma 1.2]{BCLR}}]\label{prop:borelProducts}
Let $J$ be a strongly stable ideal.
\begin{enumerate}
[(i)]
\item\label{it:borelProducts_i} Each monomial $x^\alpha$ can be written uniquely as a product $x^{\gamma}x^\delta$ with $x^\gamma \in \B_J$ and $\min x^\gamma
\geq \max x^\delta$. Therefore, $x^\delta <_\mathtt{Lex} x^\eta$ for every monomial $x^\eta$ such that $x^\eta \mid x^\alpha $ and $x^{\alpha -\eta} \notin J$.
 We will write $x^\alpha = x^\gamma \ast_J x^\delta$ to refer to this unique decomposition.
\item\label{it:borelProducts_ii} Consider $x^\alpha \in J \setminus \B_J$ and let $x_j = \min x^\alpha$. Then, $x^\alpha/x_j$ is contained in $J$.
\item\label{it:borelProducts_iii} Let $x^\beta$ be a monomial not contained in $J$. If $x^\delta x^\beta \in J$, then either $x^\delta x^\beta \in \B_J$ or
$x^\delta x^\beta = x^{\alpha}\ast_J x^{\delta'}$ with $x^\alpha \in \mathcal{B}_J$ and $x^\delta >_{\mathtt{Lex}} x^{\delta'}$. In particular, if $x_i x^\beta \in J$, 
then either $x_ix^\beta \in \B_J$ or $x_i > \min x^\beta$.
\end{enumerate}
\end{proposition}

\begin{definition}
 Let $J$ be a strongly stable ideal and let $I$ be the ideal generated by a $J$-marked set $F_J$ in $A[\xx]$. We consider the following sets of polynomials:
\begin{itemize} 
\item$ F_J^{(s)} := \left\{ x^\delta f_\alpha\ \big\vert\ \deg \big(x^\delta f_\alpha\big) = s,\ f_\alpha \in F_J,\ \min x^\alpha \geq \max x^\delta \right\}$;
\item $\widehat{F}_J^{(s)}:=\left\{ x^\delta f_\alpha\ \big\vert\ \deg \big(x^\delta f_\alpha\big) =  s,\ f_\alpha \in F_J,\ \min x^\alpha < \max x^\delta \right\}$;
\item $\textit{SF}_J^{(s)}:= \left\{x^\delta f_\beta - x^\gamma f_\alpha\ \big\vert\  x^\delta f_\beta \in \widehat{F}_J^{(s)},\ x^\gamma f_\alpha \in F_J^{(s)}, x^\delta x^\beta 
= x^\gamma x^\alpha\right\}$;
\item $\cN (J, I):= I \cap {}_A\langle \cN(J)\rangle$.  
\end{itemize} 
Throughout the paper, we use the convention that when multiplying a marked polynomial $f$ by a monomial $x^\delta$, we have $\Ht(x^\delta f) = x^\delta \Ht(f)$. Therefore, for each monomial $x^\gamma \in J_s$, there is a unique polynomial in $F_J^{(s)}$ with head term $x^\gamma$.
\end{definition}

\begin{theorem}\label{th:markedSetChar} 
Let $J$ be a strongly stable ideal and $I \subseteq A[\xx]$ be the ideal generated by a $J$-marked set $F_J$. For every $s$,
\begin{enumerate}[(i)]
\item\label{it:markedSetChar_i} $I_s = \big\langle F_J^{(s)} \big\rangle + \big\langle \widehat{F}_J^{(s)} \big\rangle = \big\langle F_J^{(s)} \big\rangle +
\big\langle \textit{SF}_J^{(s)}\big\rangle$;
\item\label{it:markedSetChar_ii} $A[\xx]_s = \big\langle F_J^{(s)}\big\rangle \oplus \big\langle \cN(J)_s \big\rangle$;
\item\label{it:markedSetChar_iii} the $A$-module $\big\langle F_J^{(s)} \big\rangle$ is free of rank equal to $\rank J_s$ and is generated by a unique $(J_s)$-marked set
$\widetilde{F}_J^{(s)}$;
\item\label{it:markedSetChar_iv} $I_s= \big\langle {F}_J^{(s)} \big\rangle \oplus  \cN (J, I)_s = \big\langle \widetilde{F}_J^{(s)} \big\rangle \oplus \cN (J, I)_s $.
\end{enumerate}
Moreover, TFAE:
\begin{enumerate}[(i)]\setcounter{enumi}{4}
\item\label{it:markedSetChar_v} $F_J$ is a $J$-marked basis;
  \item\label{it:markedSetChar_vi} for all $s$, $I_s=\big\langle {F}_J^{(s)}\big\rangle $;
\item\label{it:markedSetChar_vii} for all $s$, $\big\langle\textit{SF}_J^{(s)}\big\rangle \subseteq \big\langle {F}_J^{(s)}\big\rangle$;
\item\label{it:markedSetChar_viii} $\cN (J, I) = 0$.
\end{enumerate}
\end{theorem}

\begin{proof}
\emph{(\ref{it:markedSetChar_i})} Straightforward from the definition of the homogeneous piece of a given degree of an ideal.

\smallskip 

\emph{(\ref{it:markedSetChar_ii})} We start proving that there are no non-zero polynomials in the intersection $\langle F_J^{(s)}\rangle \cap \langle \cN(J)_s \rangle $.
Let us consider $h:=\sum_{i}  b_{i} x^{\delta_i} f_{\alpha_i}$, where $x^{\delta_i } f_{\alpha_i }$ are distinct elements of $F_J^{(s)}$ and $b_i \in A\setminus\{0\}$. 
Assume that the
polynomials  $x^{\delta_i} f_{\alpha_i}$ are indexed so that $x^{\delta_1} \geq _\Lex x^{\delta_2} \geq _\Lex \cdots$. Then   $b_{1}$ turns out to be also the 
coefficient of the monomial $x^{\delta_1} x^{\alpha_1}$ in $h$.  Indeed, $x^{\alpha_1}x^{\delta_1}$ does not appear either as head term or in the support of 
the tail of a summand $x^{\delta_i} f_{\alpha_i}$ of $h$ with $i > 1$. The monomial cannot be the head term of $x^{\delta_i} f_{\alpha_i}$, since the head terms in $F_J^{(s)}$
(and so in the summands of $h$) are all different. Moreover,  $x^{\alpha_1}x^{\delta_1}$ cannot appear in $T(x^{\delta_i} f_{\alpha_i})$ with $i>1$, since it has the unique 
decomposition $x^{\alpha_1}\ast_J x^{\delta_1}$, while every 
monomial $ x^{\delta_i} x^\beta\in T(x^{\delta_i} f_{\alpha_i})\cap J $ has decomposition $x^\alpha \ast_J x^\eta$ with $x^\eta <_\Lex x^{\delta_i}<_\Lex x^{\delta_1}$  by 
Proposition \ref{prop:borelProducts}\emph{(\ref{it:borelProducts_iii})} (note that by definition $x^\beta \in \Supp(T(f_{\alpha_1 }))\subseteq \cN(J)$). 
Therefore, no non-zero polynomials in $\langle F_J^{(s)} \rangle$ are contained in $\left\langle\cN(J)_s \right\rangle$.

To conclude the proof, we show that every monomial $x^\beta$ of degree $s$ is contained in the direct sum $\langle F^{(s)}_J \rangle \oplus \langle\cN(J)_s\rangle$. 
If $x^\beta \in \cN(J)_s$, there is nothing to prove. Now assume that there exists some monomial in $J_s$
not contained in $\langle F_{J}^{(s)}\rangle\oplus \langle\cN(J)_s\rangle$. Among them, choose $x^\beta$ such that in the unique decomposition 
$x^\beta = x^\alpha \ast_J x^\delta$,   monomial $x^\delta$ is minimum w.r.t.~the $\Lex$ ordering. Since $x^\beta = x^\delta f_{\alpha} + \T(x^\delta f_{\alpha})$,
the support of $\T(x^\delta f_{\alpha})$ cannot be contained in $\cN(J)_s$, i.e.~there exists  $x^\eta\in \Supp(T(f_{\alpha}))$ such that   $x^\eta x^\delta\in J$. By Proposition \ref{prop:borelProducts}\emph{(\ref{it:borelProducts_iii})}, we have the decomposition $x^\eta x^\delta = x^{\alpha'}\ast_J x^{\delta'}$ with 
$x^{\delta'} <_{\Lex} x^{\delta}$ against the assumption of minimality on $x^\delta$.

\smallskip

\emph{(\ref{it:markedSetChar_iii})}  By \emph{(\ref{it:markedSetChar_ii})}, we have the short exact sequence
\[
0 \rightarrow \langle F^{(s)}_J \rangle \hookrightarrow A[\xx]_s \xrightarrow{\pi} \langle \cN(J)_s \rangle\rightarrow 0.
\]
For each $x^\alpha$ in $J_s$, we compute the image $\pi(x^\alpha) = \sum_{x^\beta \in \cN(J)_s} a_{\alpha\beta} x^\beta$ and consider the set $\widetilde{F}_J^{(s)} := 
\{\tilde{f}_\alpha:=x^\alpha -  \sum_{x^\beta \in \cN(J)_s} a_{\alpha\beta} x^\beta\ \vert\ x^\alpha \in J_s\} \subseteq \ker \pi=  \langle F^{(s)}_J \rangle$. Let $J':= (J_s)$. By construction the set $\widetilde{F}_J^{(s)}$ is a
$J'$-marked set with  $\Ht(\tilde{f}_\alpha)=x^\alpha$. Applying \emph{(\ref{it:markedSetChar_ii})} to this $J'$-marked set, we have $\langle\widetilde{F}_J^{(s)}\rangle \oplus \langle\cN(J')_s\rangle 
= A[\xx]_s$. Finally, since the $A$-module generated by $\widetilde{F}_J^{(s)}$ is contained in $\langle F_J^{(s)}\rangle$ and $\cN(J)_s = \cN(J')_s$, the modules
$\langle\widetilde{F}_J^{(s)}\rangle$ and $\langle{F}_J^{(s)}\rangle$ coincide. Note that $\widetilde{F}_J^{(s)}$ is marked on the monomial ideal $J'$ generated by   $J_s$, but does not need to be a $J_{\geq s}$-marked set, since 
$J_{\geq s}$ may have minimal generators of degree $>s$. 

\smallskip

\emph{(\ref{it:markedSetChar_iv})} By \emph{(\ref{it:markedSetChar_i})} and \emph{(\ref{it:markedSetChar_iii})}, we have $I_s = \langle\widetilde{F}_J^{(s)}\rangle
+ \langle \textit{SF}_J^{(s)}\rangle$. Since $\langle \widetilde{F}_J^{(s)}\rangle \cap \langle\cN(J)_s\rangle = \{0\}$, the module   $\cN (J, F_J)_s$   can be determined starting 
from the generators of $\langle \textit{SF}_J^{(s)}\rangle$ and by replacing each monomial $x^\beta \in J_s$ appearing in 
some polynomial of $ \textit{SF}_J^{(s)}$ with the tail $\T(\widetilde{f}_\beta)$ of the polynomial $\widetilde{f}_\beta \in \widetilde{F}_J^{(s)}$ with
$\Ht(\widetilde{f}_\beta) = x^\beta$. The result of this procedure is a set of polynomials  contained both in $I_s$ and $\langle\cN(J)_s\rangle$. The sum of
$\cN (J, I)_s$ and $\langle F^{(s)}_J\rangle$ is direct by \emph{(\ref{it:markedSetChar_ii})} and \emph{(\ref{it:markedSetChar_iii})}.

\smallskip

The equivalences \emph{(\ref{it:markedSetChar_v})}$\Leftrightarrow$\emph{(\ref{it:markedSetChar_vi})}$\Leftrightarrow$\emph{(\ref{it:markedSetChar_vii})}$
\Leftrightarrow$\emph{(\ref{it:markedSetChar_viii})} follow directly from the first part of the theorem. In fact, these properties are a rephrasing of the 
definition of $J$-marked basis.
\end{proof}

We emphasize that the above result does not hold in general for a monomial ideal $J$ which is not strongly stable, as shown by the following example.

\begin{example}
 Consider $J=(x_2^2,x_1^2)\subseteq\ZZ[x_0,x_1,x_2] $ and let $I$ be the ideal generated by the $J$-marked set
 $F_J=\{f_{\text{\tiny 002}}=x_2^2+x_2x_1, f_{\text{\tiny 020}}=x_1^2+x_2x_1\}$. 
 An easy computation shows that $I_3$ is freely  generated by $F_J^{(3)}$, but
 $\vert F_J^{(3)}\vert = \rank I_3 =5 < 6 = \rank J_3$  and $I_3$ does not contain any $(J_{3})$-marked set 
 $\widetilde{F}_J^{(3)}$.
\end{example}

\begin{example}\label{ex:primoA}
Consider the strongly stable ideal $J = (x_2^2,x_2x_1,x_1^3) \subseteq \ZZ[x_0,x_1,x_2]$ and any $J$-marked set $F_J = \{f_{\text{\tiny 002}},f_{\text{\tiny 011}},f_{\text{\tiny 030}}\}$ over a ring $A$. Let us compute the sets of polynomials $F_J^{(s)}$, $\widehat{F}_J^{(s)}$ and $\textit{SF}_J^{(s)}$ discussed in Theorem \ref{th:markedSetChar} for $s=2,3,4$.
\[
\begin{split}
(s=2)\qquad &F_J^{(2)} = \{ f_{\text{\tiny 002}},\; f_{\text{\tiny 011}}\},\quad  \widehat{F}_J^{(2)} = \emptyset,\quad \textit{SF}_J^{(2)} = \emptyset,\\
(s=3)\qquad &F_J^{(3)} = \{ x_2 f_{\text{\tiny 002}},\; x_1 f_{\text{\tiny 002}},\; x_0 f_{\text{\tiny 002}},\; x_1 f_{\text{\tiny 011}},\; x_0 f_{\text{\tiny 011}},\; f_{\text{\tiny 030}}\},\\
  &\widehat{F}_J^{(3)} = \{ x_2 f_{\text{\tiny 011}}\},\quad \textit{SF}_J^{(3)} = \{x_2 f_{\text{\tiny 011}} - x_1 f_{\text{\tiny 002}}\},\\
(s=4)\qquad &F_J^{(4)} = \left\{ \begin{array}{l} x_2^2 f_{\text{\tiny 002}},\; x_2x_1 f_{\text{\tiny 002}},\; x_2x_0 f_{\text{\tiny 002}},\; x_1^2 f_{\text{\tiny 002}},\; x_1x_0 f_{\text{\tiny 002}},\; x_0^2 f_{\text{\tiny 002}},\\  x_1^2 f_{\text{\tiny 011}},\; x_1x_0 f_{\text{\tiny 011}},\; x_0^2 f_{\text{\tiny 011}},\; x_1 f_{\text{\tiny 030}},\; x_0 f_{\text{\tiny 030}}\end{array}\right\},\\
&\widehat{F}_J^{(4)} = \{x_2^2 f_{\text{\tiny 011}},\; x_2x_1 f_{\text{\tiny 011}},\; x_2x_0 f_{\text{\tiny 011}}, x_2f_{\text{\tiny 030}} \},\\
&\textit{SF}_J^{(4)} = \{x_2^2 f_{\text{\tiny 011}} - x_2x_1 f_{\text{\tiny 002}},x_2x_1 f_{\text{\tiny 011}} - x_1^2 f_{\text{\tiny 002}},x_2x_0 f_{\text{\tiny 011}} - x_1x_0 f_{\text{\tiny 002}},x_2f_{\text{\tiny 030}} - x_1^2 f_{\text{\tiny 011}} \}.\\
\end{split}
\]
In order to study the sets of polynomials $\widetilde{F}_J^{(s)}$ and the module $\cN(J,I)$, we need to know explicitly the $J$-marked set, so let us consider for instance:
\[
F_J = \left\{
f_{\text{\tiny 002}} = x_2^2 + 3x_1^2 - x_2x_0 + x_1x_0,\
f_{\text{\tiny 011}} = x_2x_1 - x_1x_0,\
f_{\text{\tiny 030}} = x_1^3 - 3 x_1^2 x_0
\right\}
\]
and let $I := (F_J)$. For $s=2$, we have $\widetilde{F}_J^{(2)} = F_J^{(2)}$ and $\cN(J,I) = \emptyset$.

$(s=3)$. In order to construct $\widetilde{F}_J^{(3)}$, we have to determine the equivalence classes of monomials in the quotient $A[x_0,x_1,x_2]_3/\langle F_J^{(3)}\rangle \simeq \langle \cN(J)_3\rangle$. If $h\in A[x_0,x_1,x_2]_s$, we denote by $\overline{h}$ its class  in $A[x_0,x_1,x_2]_s/\langle F_J^{(s)}\rangle$. Following the strategy of the proof of Theorem \ref{th:markedSetChar}, we examine the monomials of $J_3$ in increasing order w.r.t.~the $\Lex$ ordering.
\[
\begin{split}
& \overline{x_1^3} \stackrel{-f_{\text{\tiny 030}}}{=} \overline{3 x_1^2 x_0} \quad \Rightarrow \quad\widetilde{f}_{\text{\tiny 030}} = f_{\text{\tiny 030}},\\
&\overline{x_2x_1x_0} \stackrel{-x_0f_{\text{\tiny 011}}}{=} \overline{x_1 x_0^2} \quad \Rightarrow \quad\widetilde{f}_{\text{\tiny 111}} = x_0f_{\text{\tiny 011}},\\
&\overline{x_2x_1^2} \stackrel{-x_1f_{\text{\tiny 011}}}{=} \overline{x_1^2 x_0}  \quad \Rightarrow \quad\widetilde{f}_{\text{\tiny 021}} = x_1f_{\text{\tiny 011}},\\
&\overline{x_2^2x_0} \stackrel{-x_0f_{\text{\tiny 002}}}{=} \overline{-3x_1^2 x_0 + x_2x_0^2 -x_1x_0^2}  \quad \Rightarrow \quad\widetilde{f}_{\text{\tiny 102}} = x_0f_{\text{\tiny 002}},\\
&\overline{x_2^2x_1} \stackrel{-x_1f_{\text{\tiny 002}}}{=} \overline{-3x_1^3 + x_2x_1x_0 -x_1^2x_0} \stackrel{3\widetilde{f}_{\text{\tiny 030}}}{=} \overline{x_2x_1x_0 -10 x_1^2x_0} = {}\\
&\phantom{\overline{x_2^2x_1}} \stackrel{-\widetilde{f}_{\text{\tiny 111}}}{=} \overline{-10 x_1^2x_0 + x_1x_0^2} \quad \Rightarrow \quad\widetilde{f}_{\text{\tiny 012}} = x_1f_{\text{\tiny 002}} - 3f_{\text{\tiny 030}} + x_0 f_{\text{\tiny 011}},\\
&\overline{x_2^3} \stackrel{-x_2f_{\text{\tiny 002}}}{=} \overline{-3x_2x_1^2 + x_2^2x_0 -x_2x_1x_0} \stackrel{3\widetilde{f}_{\text{\tiny 021}}}{=} \overline{x_2^2x_0 -x_2x_1x_0 -3x_1^2x_0} = {}\\
&\phantom{\overline{x_2^2x_1}} \stackrel{-\widetilde{f}_{\text{\tiny 102}}}{=} \overline{-x_2x_1x_0-6x_1^2x_0+x_2x_0^2-x_1x_0^2} = {}\\
&\phantom{\overline{x_2^2x_1}} \stackrel{\widetilde{f}_{\text{\tiny 111}}}{=} \overline{-6x_1^2x_0+x_2x_0^2-2x_1x_0^2} \quad \Rightarrow \quad\widetilde{f}_{\text{\tiny 003}} = x_2f_{\text{\tiny 002}} - 3x_1f_{\text{\tiny 011}} + x_0 f_{\text{\tiny 002}} - x_0f_{\text{\tiny 011}}.\\
\end{split}
\]
To determine $\cN(J,I)_3$, we can compute the class of the polynomial of $\textit{SF}_J^{(3)}$ in $A[x_0,x_1,x_2]_3/\langle F_J^{(3)}\rangle = A[x_0,x_1,x_2]_3/\langle \widetilde{F}_J^{(3)}\rangle$:
\[
\overline{x_2 f_{\text{\tiny 011}} - x_1 f_{\text{\tiny 002}}} = \overline{-3x_1^3-x_1^2x_0} \stackrel{3\widetilde{f}_{\text{030}}}{=} \overline{-10x_1^2x_0}
\]
so that $\cN(J,I)_3 = \langle 10 x_1^2x_0\rangle$.

\smallskip

$(s=4)$. Repeating the same procedure applied for $s=3$, we obtain:
\small
\[
\begin{split}
 \widetilde{f}_{\text{\tiny 130}} &{}= x_0 {f}_{\text{\tiny 030}} = x_1^3x_0 -3 x_1^2x_0^2,\\
 \widetilde{f}_{\text{\tiny 040}} &{}= x_1 {f}_{\text{\tiny 030}} - bx_0{f}_{\text{\tiny 030}} = x_1^4 - 9 x_1^2x_0^2,\\
 \widetilde{f}_{\text{\tiny 211}} &{}= x_0^2{f}_{\text{\tiny 011}} = x_2x_1x_0^2 -  x_1x_0^3,\\
 \widetilde{f}_{\text{\tiny 121}} &{}= x_1x_0{f}_{\text{\tiny 011}} = x_2x_1^2x_0 -  x_1^2x_0^2,\\
 \widetilde{f}_{\text{\tiny 031}} &{}= x_1^2 {f}_{\text{\tiny 011}} + x_0 {f}_{\text{\tiny 030}} = x_2x_1^3 -3 x_1^2x_0^2,\\
 \widetilde{f}_{\text{\tiny 202}} &{}= x_0^2{f}_{\text{\tiny 002}} = x_2^2x_0^2 + 3x_1^2x_0^2 - x_2x_0^3 +x_1x_0^3,\\
 \widetilde{f}_{\text{\tiny 112}} &{}= x_1x_0{f}_{\text{\tiny 002}} - 3x_0{f}_{\text{\tiny 030}} + x_0^2 {f}_{\text{\tiny 211}} = x_2^2x_1x_0 + 10 x_1^2x_0^2 - x_1 x_0^3,  \\
 \widetilde{f}_{\text{\tiny 022}} &{}= x_1^2 {f}_{\text{\tiny 002}} - 3x_1 {f}_{\text{\tiny 030}} + x_1x_0 {f}_{\text{\tiny 011}} - 10 x_0 {f}_{\text{\tiny 030}} = x_2^2x_1^2 + 29 x_1^2x_0^2, \\
 \widetilde{f}_{\text{\tiny 103}} &{}= x_2x_0 {f}_{\text{\tiny 002}} - 3x_1x_0{f}_{\text{\tiny 011}} + x_0^2 {f}_{\text{\tiny 002}} - x_0^2{f}_{\text{\tiny 011}} = x_2^3x_0 + 6x_1^2x_0^2 - x_2x_0^3 + 2x_1x_0^3,    \\
\widetilde{f}_{\text{\tiny 013}} &{}= x_2x_1 {f}_{\text{\tiny 002}} -3x_1^2{f}_{\text{\tiny 011}} + x_1x_0 {f}_{\text{\tiny 002}} - x_1x_0{f}_{\text{\tiny 011}} -6 x_0 {f}_{\text{\tiny 030}} + x_0^2 {f}_{\text{\tiny 011}} = x_2^3x_1 + 20 x_1^2x_0^2 - x_1x_0^3,     \\
 \widetilde{f}_{\text{\tiny 004}} &{}= x_2^2{f}_{\text{\tiny 002}} - 3x_1^2{f}_{\text{\tiny 002}} + 9x_1{f}_{\text{\tiny 030}} + x_2x_0{f}_{\text{\tiny 002}} - x_1x_0{f}_{\text{\tiny 002}} - 6 x_1x_0 {f}_{\text{\tiny 011}} +33 x_0{f}_{\text{\tiny 030}} + x_0^2 {f}_{\text{\tiny 002}} - 2x_0^2 {f}_{\text{\tiny 011}} = {}\\
&{} = x_2^4 -91x_1^2x_0^2 - x_2x_0^3+3x_1x_0^3.       \\
\end{split}
\]
\normalsize
Moreover,  $\cN(J,I)_4 = \left\langle 10 x_1^2x_0^2\right\rangle$ as in the quotient $A[x_0,x_1,x_2]_4/\langle F_J^{(4)}\rangle$ we have
\[
\begin{split}
&\overline{x_2^2 f_{\text{\tiny 011}} - x_2x_1 f_{\text{\tiny 002}}} = \overline{-10 x_1^2x_0^2},\qquad
\overline{x_2x_1 f_{\text{\tiny 011}} - x_1^2 f_{\text{\tiny 002}}} = \overline{-30 x_1^2x_0^2},\\
&\overline{x_2x_0 f_{\text{\tiny 011}} - x_1x_0 f_{\text{\tiny 002}}} = \overline{-10 x_1^2x_0^2},\qquad 
\overline{x_2f_{\text{\tiny 030}} - x_1^2 f_{\text{\tiny 011}}} = \overline{0}.
\end{split}
\]
\normalsize
Therefore, $F_J$ is not a $J$-marked basis unless $10 = 0$ in $A$ (cf.~Theorem \ref{th:markedSetChar}).
\end{example}
We conclude this section giving a characterization of a $J$-marked basis $F_J$ that takes into account the homogeneous pieces $(F_J)_s$ of the ideal it generates for a
limited number of degrees. The following statement is clearly inspired by Gotzmann's Persistence Theorem, but we emphasize that the proof is independent from that result. 

\begin{theorem}\label{th:rifatto}
Let $J$ be a strongly stable ideal, $m$ be the maximum degree of monomials in its minimal monomial basis $\B_J$ and $I$ be the ideal in $A[\xx]$ generated by 
a $J$-marked set $F_J$. TFAE:
\begin{enumerate}[(i)]
\item\label{it:rifatto_i} $F_J$ is a $J$-marked basis;
\item\label{it:rifatto_ii} as an $A$-module, $I_{s} = \langle{F}^{(s)}_J\rangle$ for every $s \leq m+1$;
\item\label{it:rifatto_iii} as an $A$-module, $I_{s} = \langle\widetilde{F}^{(s)}_J\rangle$ for every $s \leq m+1$;
\item\label{it:rifatto_iv}   $\cN(J,I)_s=0$ for every $s \leq m+1$.
\end{enumerate}
\end{theorem}
\begin{proof} 
\emph{(\ref{it:rifatto_i})}$\Rightarrow$\emph{(\ref{it:rifatto_ii})} Straightforward by Theorem \ref{th:markedSetChar}\emph{(\ref{it:markedSetChar_vi})}.

\smallskip

\emph{(\ref{it:rifatto_ii})}$\Rightarrow$\emph{(\ref{it:rifatto_i})} We want to prove that for every $s$, $A[\xx]_s = I_s \oplus \langle\cN(J)_s\rangle$. 
This is true for $s\leq m+1$ by hypothesis. By Theorem \ref{th:markedSetChar}\emph{(\ref{it:markedSetChar_ii})-(\ref{it:markedSetChar_iii})}, we  know that $A[\xx]_s 
= \langle {F}_J^{(s)}\rangle \oplus \langle\cN(J)_s\rangle$ and $\langle {F}_J^{(s)}\rangle \subseteq I_s$, so that we need to prove $I_s
\subseteq \langle {F}_J^{(s)}\rangle$.
Let us assume that this is not true and let $t$ be the minimum degree for which $I_t \not \subseteq \langle {F}_J^{(t)}\rangle$. Note that $t \geq m+2 > m$ 
and $I_t = x_0I_{t-1} + \cdots + x_nI_{t-1}$.

Since $I_{t-1}= \langle{F}_J^{(t-1)}\rangle$,  there should exist a variable $x_i$  such that $x_i I_{t-1}\not \subseteq \langle 
{F}_J^{(t)}\rangle$, or equivalently  $x_i  {F}_J^{(t-1)} \not \subseteq \langle{F}_J^{(t)}\rangle$.
Assume that $x_i$ has the minimum index  and take a polynomial 
$ x^\delta f_\alpha \in {F}_J^{(t-1)}$, with $x^\alpha=\Ht(f_\alpha)  \in \B_J$,  such that $x_i  x^\delta f_\alpha\notin  \langle{F}_J^{(t)}\rangle $. 
The variable  $x_i$ has to be greater than $\min x^\alpha$, since otherwise  $x_i x^\delta f_\alpha \in  {F}_J^{(t)}$. Moreover, $\vert \delta \vert >0 $
since $t-1>m$.   Let $x_j =
\max x^\delta \leq \min x^\alpha < x_i$ and $x^{\delta'} = \frac{x^\delta}{x_j}$. The polynomial $x_i x^{\delta'} f_\alpha$ is contained in 
$ I_{t-1}$, while  $x_j(x_i x^{\delta'} f_\alpha)=x_ix^\delta f_\alpha $ is not  contained in  $ \langle{F}_J^{(t-1)}\rangle$, contradicting the 
minimality of $i$.

\smallskip

\emph{(\ref{it:rifatto_ii})}$\Leftrightarrow$\emph{(\ref{it:rifatto_iii})}$\Leftrightarrow$\emph{(\ref{it:rifatto_iv})}  Straightforward by  Theorem \ref{th:markedSetChar}.
\end{proof}

\section{Definition and representability of marked functors}

We follow the notation for functors used in \cite{HaimSturm}. 
 The main object of interest in the present paper is   the set
\begin{equation}\label{eq:markedfamily}
\MFFunctor{J}(A) := \big\{ \text{ideals }I \subseteq A[\xx]\ \vert\ A[\xx] = I \oplus {}_A\langle \cN(J)\rangle\big\}
\end{equation}
which is defined for every   noetherian ring $A$ and every strongly stable ideal  $J\subseteq A[\xx]$.   In this  section we will prove that 
 this construction is in fact functorial, i.e.~$\MFFunctor{J}(A)$ is the evaluation in the noetherian ring $A$ of a functor
\begin{equation*}
\MFFunctor{J}: \underline{\textnormal{Noeth-Rings}} \rightarrow \Sets.
\end{equation*}

Now we will describe the elements of any $\MFFunctor{J}(A)$ in terms of the notion of  $J$-marked basis, discussed in the previous section. 
This will be a key point to prove 
its functoriality.

\begin{proposition}\label{prop:markedSetExists}
Let $J$ be a strongly stable ideal  and let $I$ be an element of $\MFFunctor{J}(A)$. 
\begin{enumerate}[(i)]	
\item\label{it:markedSetExists_i} The ideal $I$ contains a  unique  $J$-marked set $F_J$.
\item\label{it:markedSetExists_ii} $I = (F_J)$ and $F_J$ is the unique $J$-marked basis  contained in $I$.
\end{enumerate}
\end{proposition}
\begin{proof}
\emph{(\ref{it:markedSetExists_i})} Let $x^\alpha$ be a minimal generator of $J$ and consider its image by the projection $A[\xx] \xrightarrow{\pi_I} A[\xx]/I$. Since
$A[\xx]_{\vert\alpha\vert}/I_{\vert\alpha\vert} \simeq \langle \cN(J)_{\vert\alpha\vert}\rangle$, $\pi_I(x^\alpha)$ is given by a linear
combination $\sum c_{\alpha\beta} x^\beta$ of the monomials $x^\beta \in \cN(J)_{\vert\alpha\vert}$. Therefore, $\ker \pi$ contains a unique homogeneous polynomial $f_{\alpha} = x^\alpha - \sum c_{\alpha\beta} x^\beta$ with head term $x^\alpha$. The collection of all $f_{\alpha}$, for $x^\alpha \in \mathcal{B}_J$, is the unique $J$-marked set. 

 \emph{(\ref{it:markedSetExists_ii})}  Starting from  $ F_J$ we can construct, for every  degree $s$, the sets of polynomials  $ F_J^{(s)}$ and $\widetilde  F_J^{(s)}$
 as in Theorem \ref{th:rifatto}. Recall that they are both contained in the ideal $(F_J)\subseteq I$.   In order to show that $F_J$ is a $J$-marked basis and  generates $I$, we observe  that for every 
 $s$, $\langle F_J^{(s)}\rangle \subseteq I_s$ and 
$\langle F_J^{(s)}\rangle \oplus \langle\cN(J)_s\rangle = A[\xx]_s $   by Theorem \ref{th:markedSetChar}\emph{(\ref{it:markedSetChar_ii})}. Moreover $I_s \oplus \langle\cN(J)_s\rangle
= A[\xx]_s $,  since $I\in \MFFunctor{J}(A)$. Therefore, $I_s=\langle F_J^{(s)}\rangle=(F_J^{(s)})_s$ in every degree $s$.
Finally, $F_J$ is a $J$-marked basis by Theorem \ref{th:markedSetChar}\emph{(\ref{it:markedSetChar_v}})-\emph{(\ref{it:markedSetChar_vi})}  and  is unique by \emph{(\ref{it:markedSetExists_i})}.
\end{proof}

\begin{remark}\label{rk:conesempio}
We emphasize that uniqueness is not true for a $J$-marked set generating an ideal $I \notin \MFFunctor{J}(A)$. For instance, consider the strongly 
stable ideal $J = (x_2^2,x_2x_1,x_1^3) \subseteq \ZZ[x_0,x_1,x_2]$. The $J$-marked set $F_J = \{x_2^2 + x_0^2, x_2x_1,x_1^3\}$ defines an ideal $I = (F_J)$ not contained
in $\MFFunctor{J}(\ZZ)$ as $x_1x_0^2 = x_1(x_2^2+x_0^2) - x_2 (x_2x_1) \in I\cap \cN(J)$.
In fact, the ideal $I$ is generated by infinitely many $J$-marked sets $\{x_2^2 + x_0^2, x_2x_1,x_1^3 + a\,x_1x_0^2\},\ a\in \ZZ$.
\end{remark}

As a consequence of the previous result, we are now able to  give a new description  of  $\MFFunctor{J}(A)$: 
\[
\MFFunctor{J}(A)= \left\{ \text{ideals }I \subseteq A[\xx ] \ \vert\ I  \text{ is generated by a } J\text{-marked basis} \right\}.
\]
For every strongly stable ideal $J$, let us consider the map between the category 
of noetherian rings to the category of sets
\begin{equation}\label{eq:markedFunctor}
\MFFunctor{J}: \underline{\text{Noeth-Rings}} \rightarrow \Sets
\end{equation}
that associates to a noetherian ring $A$ the set $\MFFunctor{J}(A)$ and to a morphism $\phi: A \rightarrow B$ the map
\begin{equation}\label{eq:markedFamilyMap}
\begin{split}
\MFFunctor{J}(\phi):\ \MFFunctor{J}(A)\ &\longrightarrow\ \MFFunctor{J}(B)\\
\parbox{1.5cm}{\centering $I$}\ & \longmapsto\ I \otimes_A B.
\end{split}
\end{equation}

\begin{proposition}\label{prop:mfFunctor}
For every strongly stable ideal $J$, $\MFFunctor{J}$ is a functor.
\end{proposition}
\begin{proof}
Consider the $J$-marked basis $F_{J,A}$ generating the ideal $I \in \MFFunctor{J}(A)$. Any morphism $\phi: A \rightarrow B$ gives the structure of $A$-module to $B$. 
Thus, tensoring $I$ by $B$ leads to the following transformation on the $J$-marked basis $F_{J,A}$:
\[
f_{\alpha,A} = x^\alpha - \sum c_{\alpha\beta} x^\beta \in F_{J,A} \quad \longmapsto\quad f_{\alpha,B} = x^\alpha - \sum \phi(c_{\alpha\beta}) x^\beta.
\]
since $\phi(1_A) = 1_B$, the set $F_{J,B} := \{f_{\alpha,B}\ \vert\ f_{\alpha,A} \in F_{J,A}\}$ is still a $J$-marked  set. Finally, $F_{J,B}$ is a
$J$-marked basis since the tensor product by $\otimes_A B$ of a direct sum of free $A$-modules is a direct sum of free $B$-modules.
\end{proof}

Now we discuss a necessary condition for this functor to be representable. 
\begin{lemma}\label{lem:mfZariskiSheaf}
For every strongly stable  ideal $J$, $\MFFunctor{J}$ is a Zariski sheaf.
\end{lemma}
\begin{proof}
Let $A$ be a noetherian ring and $U_i = \Spec A_{a_i}$, $i=1,\dots, s$, an open cover of $\Spec A$, which is equivalent to require that $(a_1, \dots, a_s)=A$. Consider a set of ideals 
$I_i \in \MFFunctor{J}(A_{a_i})$ such that for any pair of
indices $i \neq j$
\[
I_{ij} := I_i \otimes_{A_{a_i}} A_{a_i a_j} = I_j \otimes_{A_{a_j}} A_{a_i a_j} \in \MFFunctor{J}(A_{a_i a_j}).
\]
We need to show that there exists a unique ideal $I \in \MFFunctor{J}(A)$ such that $I_i = I \otimes_A A_{a_i}$ for every $i$.

Let us consider the $J$-marked bases associated to $I_i$:
\[
F_{J,i} = \left\{ f_{\alpha,i} = x^\alpha - \sum_{\mathclap{\cramped{x^\beta \in \cN(J)_{\vert\alpha\vert}}}} c_{\alpha\beta}^{(i)}\, x^\beta \in A_{a_i}[\xx]\ \; \Bigg\vert\ \; x^\alpha \in \mathcal{B}_J\right\},\quad I_i = (F_{J,i}) \subseteq A_{a_i}[\xx],\quad \forall\ i=1,\ldots,s.
\]
By assumption, for each $x^\alpha \in \B_J$ and for each pair of indices $i\neq j$, the polynomials $f_{\alpha,i}$ and $f_{\alpha,j}$ coincide on $A_{a_i a_j}[\xx]$. By the 
sheaf axiom for the quasi-coherent sheaf $\widetilde{A[\xx]}$ on $\Spec A$, we know that there exists a unique polynomial $f_{\alpha} \in A[\xx]$ whose image in
$A_{a_i}[\xx]$ 
is $f_{\alpha,i}$ for every $i$. The polynomial $f_{\alpha}$ turns out to be monic.  In fact, if $c$ is the coefficient of $x^\alpha$, then its image
in $A_{a_i}$ is 
$1_{A_{a_i}}$, so that  $(c-1_A)a_i^k=0$ for some integer $k$. Thus, $c=1_A$ since $(a_1^k, \dots, a_s^k)=A$.  
The collection of polynomials $\{f_\alpha\ :\ x^\alpha \in \B_J\}$ forms a $J$-marked basis.
\end{proof}

Now we prove that the functor $\MFFunctor{J}$ is representable finding explicitly the affine scheme $\MFScheme{J}$  representing it. To do that
we apply  the previous theorems that describe  which
conditions on the coefficients of polynomials in a $J$-marked set guarantee that the marked set is a $J$-marked basis.

  We  obtain $\MFScheme{J}$  as a closed subscheme of an affine scheme of a  suitable dimension depending on $J$.
  
\medskip

\begin{notation}\label{notazioni} Let $J$ be any  strongly stable monomial ideal in $ \ZZ[\xx]$. Then:
\begin{itemize}
 \item $\CC$ is the  set of variables of the coordinate ring of the affine scheme $\mathbb{A}^N = \Spec \ZZ[\CC]$, where 
\[
N = \sum_{x^\alpha \in \mathcal{B}_J} \left\vert \cN(J)_{\vert\alpha\vert}\right\vert.
\]
We consider the variables in $\CC$ indexed as $C_{\alpha\beta}$ where the multi-index $\alpha$ corresponds to $x^\alpha \in \mathcal{B}_J$ and the multi-index $\beta$ to $x^\beta \in \cN(J)_{\vert\alpha\vert}$.
\item $\mathcal I$ is the ideal in $\ZZ[\CC][\xx]=\ZZ[\CC]\otimes_\ZZ \ZZ[\xx]$ generated by the following  $J$-marked set 
\begin{equation}\label{eq:MarcataParametri}
  \mathcal{F}_J := \left\{ x^\alpha - \sum_{\mathclap{\cramped{x^\beta \in \cN(J)_{\vert\alpha\vert}}}} C_{\alpha\beta}\, 
  x^\beta \in \ZZ[\CC][\xx]\ \bigg\vert\
  x^\alpha \in \B_J\right\}.
  \end{equation}
\item $ \mathfrak{I}_J$    the ideal in $\ZZ[\CC]$  generated by the $\xx$-coefficients 
of the polynomials in $\cN(J,\mathcal{I})$.

\item  Every $J$-marked set $F_J=\{f_\alpha = x^\alpha - \sum_{x^\beta \in \cN(J)_{\vert\alpha\vert}} c_{\alpha\beta}\, x^\beta \ \vert\  x^\alpha \in \B_J\}$  in  $A[\xx]$ 
is uniquely identified by the coefficients $c_{\alpha\beta}\in A$, or equivalently, by the ring homomorphism 
\begin{equation*}
\phi_{F_J} : \ZZ[\CC] \rightarrow A  : C_{\alpha\beta}\mapsto c_{\alpha\beta}. 
\end{equation*}
\item Moreover, let $\phi_{F_J}[\xx]\colon \ZZ[\CC][\xx] \rightarrow A[\xx]$ be the canonical extension of $\phi_{F_J}$.
\end{itemize}
\end{notation}

\begin{theorem}\label{th:mfRepresentable}
In the notation above, the functor $\MFFunctor{J}$ is represented by $\MFScheme{J} := \Spec \ZZ[\CC]/\mathfrak{I}_J$. 
Therefore, a $J$-marked set  $F_J$ is a $J$-marked basis if, and only if, 
 $\phi_{F_J}$ factors.
 \begin{center}
\begin{tikzpicture}
\node (a) at (0,0) [] {$\ZZ[\CC]$};
\node (b) at (3,0) [] {$A$};
\node (c) at (1.5,-1.25) [] {$\ZZ[\CC]/\mathfrak{I}_{J}$};

\draw [->] (a) -- (b);
\draw [->>] (a) -- (c);
\draw [->] (c) -- (b);
\end{tikzpicture}
\end{center}
\end{theorem}
\begin{proof}
Let $A$ be any noetherian ring, $F_J$ be a $J$-marked set in $A[\xx]$ and $I = (F_J) \subseteq A[\xx]$.
We obtain the statement proving that $F_J$ is a $J$-marked basis if, and only if, $\ker \phi_{F_J}\supseteq \mathfrak{I}_J$. 
 
By definition, $\phi_{F_J}[\xx] $ is the identity on monomials and $\phi_{F_J}[\xx](\mathcal I)\subseteq I$, so that 
$\phi_{F_J}[\xx] \left(\cN(J,\mathcal I)\right)\subseteq \cN(J,I)$. If $F_J$ is a $J$-marked basis, then  $\cN(J,I)=0 $, hence  $\ker( \phi_{F_J})\supseteq \mathfrak{I}_J$.
 
 On the other hand, if $\ker( \phi_{F_J})\supseteq \mathfrak{I}_J$, then  for every $s$ we have  
 $$\widehat{F}_J^{(s)}=\phi_{F_J}[\xx](\widehat{\mathcal{F}}_J^{(s)}) \subseteq  \phi_{F_J}[\xx](\mathcal{I}_s)=
 \phi_{F_J}[\xx]\big( \langle  \mathcal{F}_J^{(s)}\rangle \oplus \cN(J,\mathcal{I})_s\big)= \phi_{F_J}[\xx]( \langle  \mathcal{F}_J^{(s)}\rangle )\subseteq  
 \left._A\langle F^{(s)}\rangle\right.$$
 so that $ I_s={}_A \langle  {F}_J^{(s)}\cup  \widehat{F}_J^{(s)}\rangle \subseteq {}_A\langle {F}_J^{(s)}\rangle \subseteq I_s$ 
 and we conclude applying Theorem \ref{th:markedSetChar}\emph{(\ref{it:markedSetChar_v})-(\ref{it:markedSetChar_vi})}. 
\end{proof}

In order to explicitly compute a finite set of generators of the ideal $\mathfrak{I}_J$, we can apply some of the previous results. 
By Theorem \ref{th:rifatto} we get the following simplification.

\begin{corollary}
 For every strongly stable ideal $J$, the ideal $\mathfrak{I}_J$ is generated by the $\xx$-coefficients 
of the polynomials in $\cN(J,\mathcal{I})_s$ for every $s\leq m+1$, where $m$ is the maximum degree of monomials in $\B_J$.
\end{corollary}
\begin{proof}
 Let $\mathfrak{I'}$ the ideal in $\ZZ[\CC]$  generated by $\xx$-coefficients 
of the polynomials in $\cN(J,\mathcal{I})_s$ for every $s\leq m+1$. Obviously $\mathfrak{I'} \subseteq \mathfrak{I}_J$.

On the other hand, by Theorem \ref{th:rifatto}, the image of $\mathcal{F}_J$ in $\ZZ[\CC]/\mathfrak{I'}\otimes_\ZZ \ZZ[\xx]$ is a $J$-marked
basis. Therefore the map $\ZZ[\CC]\rightarrow \ZZ[\CC]/\mathfrak{I'}$ factors through  $\ZZ[\CC]/ \mathfrak{I}_J$.
\end{proof}

We can obtain a set of generators of $\mathfrak{I}_J$ by computing  a set of generators of the  $\ZZ$-module  $\cN (J,\mathcal{F}_J)_s$ for each  $s\leq m+1$ through a Gaussian reduction. This is the method applied, for instance, in \cite{CioffiRoggero}.

A more efficient method is the one developed in \cite{BCLR}, which is similar to the  Buchberger algorithm for  Gr\"obner bases. We will know describe this 
method and then prove that it gives in fact a set of generators of   $\mathfrak{I}_J$, as claimed in  \cite{BCLR}.

\begin{definition}\label{riduzione} 
Let $J$ be a strongly stable ideal and let $F_J$ be a marked set. We say that a polynomial $g$ is a \emph{$J$-remainder} if $\Supp(g) \cap J = \emptyset$.
Given two polynomials $h$ and $g$, we say that $g$ can be obtained from $h$ by a step of \emph{$F_J$-reduction} if $g = h - c\, f $ where $c$ is the coefficient in $h$ of a monomial $x^\eta \in J_s$ and  $f $ is the unique polynomial in $ F_J^{(s)}$ with   $\Ht(f) = x^\eta$.   We write 
\[
 h  \xrightarrow{F_J^{(\cdot)}} g
\]
if $g$ arises from $h$  by a finite sequence of reductions as described above. Moreover, we write $h \xrightarrow{F_J^{(\cdot)}}_\ast g$ if $g$ is a $J$-remainder.
\end{definition}

As proved in \cite{CioffiRoggero,BCLR}, the procedure $\xrightarrow{F_J^{(\cdot)}}$ is noetherian, i.e.~every sequence of $F_J$-reductions starting on a polynomial $h$ stops after
a finite number of steps giving   a $J$-remainder polynomial $g$. Indeed, each step of reduction $h \mapsto h - c f$ replaces a monomial $x^\eta=x^\alpha \ast_J x^\gamma$
in the support of $h$ 
with $x^\gamma T(f_\alpha)$,   $f_\alpha \in F_J$. Since $T(f_\alpha) \in \langle \cN(J)\rangle$,  every monomial appearing in $x^\gamma T(f_\alpha)$ either is in $\cN(J)$ 
or has the decomposition  $x^\gamma x^\beta=x^{\alpha'}\ast_J x^\delta$ with $x^\delta <_{\Lex}x^\gamma $. This allows to conclude since  $<_{\Lex}$ is a well-ordering 
on monomials. 

We can find many different sequences of steps of 
reduction starting from a given  polynomial $h$, but the $J$-remainder polynomial $g$ is unique.
In fact, if $h$ is a homogeneous polynomial of degree $s$ and $h \xrightarrow{F_J^{(\cdot)}}_\ast g$ and  $h \xrightarrow{F_J^{(\cdot)}}_\ast g'$, then $g-g'=(g-h)-(g'-h)
\in \langle F_J^{(s)} \rangle$ since $g-h,g'-h \in \langle F_J^{(s)} \rangle$. By definition of $J$-remainder,  $g-g' \in \langle\cN(J)_s\rangle$ and $\langle\cN(J)_s\rangle \cap \langle F_J^{(s)}
\rangle = \{0\}$ by Theorem \ref{th:markedSetChar}\emph{(\ref{it:markedSetChar_ii})}.

\begin{remark}
In general, the marking cannot be performed w.r.t.~a term ordering (see \cite[Example 3.18]{CioffiRoggero}), so that the noetherianity of the procedure is surprising. Indeed, it is well-known that a general reduction process by a set of marked polynomials is noetherian if, and only if, the marking is performed w.r.t.~a term ordering  (see \cite{ReevesSturmfels}). The ultimate reason for this is our restriction that each monomial $x^\eta \in J$ is reduced by the unique polynomial $x^\gamma f_\alpha \in F_J^{(s)}$ such that $x^\eta = x^\alpha \ast_J x^\gamma$ --- as opposed to \emph{any} polynomial $x^\delta f_\beta \in (F_J)$ such that $x^\delta x^\beta = x^\eta$.
\end{remark}

%
%

Now we can give a characterization of $J$-marked basis in terms of this reduction procedure and $S$-polynomials.
\begin{definition}
For marked polynomials $f_\alpha,f_\beta$ in a $J$-marked set, we call $S(f_\alpha,f_\beta):= x^\eta f_\alpha - x^\nu f_\beta$ the \emph{$S$-polynomial} of
$f_\alpha$ and $f_\beta$, where $(x^\eta,-x^\nu)$ is the minimal  syzygy between $x^\alpha$ and $x^\beta$. We call \emph{$\textsc{EK}$-polynomial},  and denote by 
$S^{\EK}(f_\alpha,f_\beta)$, a $S$-polynomial whose syzygy $(x^\eta,-x^\nu)$ is of Eliahou-Kervaire type, i.e.~$x^\eta$ is a single variable $x_j$ greater 
that $\min(x^\alpha)$ and $x_j x^\alpha = x^\beta\ast_J x^\nu$ (see \cite{EliahouKervaire}).\end{definition}

 Notice that for every $\textsc{EK}$-polynomial  $ x_j f_\alpha-x^\nu f_\beta \in A[\xx]_s$,  we have $x_j f_\alpha\in \widehat{F}_J^{(s)}$ and $x^\nu f_\beta\in F^{(s)}_J$.

\begin{theorem}\label{th:markedBasisCharacterization}
Let $J$ be a strongly stable ideal and let $F_J$ be a $J$-marked set. TFAE:
\begin{enumerate}[(i)]
\item\label{it:markedBasisCharacterization_i} $F_J$ is a $J$-marked basis;
\item\label{it:markedBasisCharacterization_iii} $S^{\EK}(f_\alpha,f_\beta) \xrightarrow{F_J^{(\cdot)}}_\ast 0$, for all $S^{\EK}(f_\alpha,f_\beta)$ with
$f_\alpha,f_\beta \in F_J$;
\item\label{it:markedBasisCharacterization_ii} $x_i f_\alpha \xrightarrow{F^{(\cdot)}_J}_\ast 0$, for all $f_\alpha \in F_J$ and $x_i > \min x^\alpha$.
\end{enumerate}
\end{theorem}
\begin{proof}
\text{\emph{(\ref{it:markedBasisCharacterization_i})}}$\Leftrightarrow$\text{\emph{(\ref{it:markedBasisCharacterization_iii})}} The $\text{EK}$-syzygies are
a basis of the syzygies of $J$, so that $\emph{(\ref{it:markedBasisCharacterization_iii})}$ ensures that every other syzygy between two monomials $x^\alpha,x^\beta 
\in \mathcal{B}_J$ lifted to the corresponding marked polynomials $f_\alpha$ and $f_\beta$ has $J$-remainder equal to $0$. Since these syzygies are exactly the generators
of the module $\textit{SF}_J^{(s)}$ for all $s$, $\emph{(\ref{it:markedBasisCharacterization_iii})}$ is equivalent to $(F_J)_s = \langle F^{(s)}_J\rangle$ for every $s$,   hence to  
 $\text{\emph{(\ref{it:markedBasisCharacterization_i})}}$ by  Theorem \ref{th:markedSetChar}.

\smallskip

\text{\emph{(\ref{it:markedBasisCharacterization_ii})}}$\Leftrightarrow$\text{\emph{(\ref{it:markedBasisCharacterization_iii})}} The two reductions agree.
\end{proof}

\begin{remark}
Notice that it is possible to prove the equivalence between Theorem \ref{th:markedBasisCharacterization}\emph{(\ref{it:markedBasisCharacterization_i})} and 
 \emph{(\ref{it:markedBasisCharacterization_ii})} directly from the properties of the reduction $\xrightarrow{F^{(\cdot)}_J}_\ast$,
as was done in \cite{BCLR}. 
\end{remark}

We now show how  a set of generators of  $\mathfrak{I}_J$ can be computed using Theorem \ref{th:markedBasisCharacterization}.
We consider the $J$-marked set ${\mathcal{F}}_{J}$ given in (\ref{eq:MarcataParametri}) and use the marked sets $\widetilde{\mathcal{F}}_{J}^{(s)}$ in order to perform the
polynomial reduction in each degree $s$. 
The elements of 
$\widetilde{\mathcal{F}}_{J}^{(s)}$ take the shape
\[
\widetilde{f}_{\gamma} = x^\gamma - \sum_{x^\delta \in \cN(J)_s} D_{\gamma\delta}\, x^\delta,\qquad \forall\ x^\gamma \in J_s
\]
with  coefficients $D_{\gamma\delta}\in \ZZ[\CC]$. 

Let $S^{\EK}(f_{\alpha},f_{\alpha'})$ be an $\textnormal{EK}$-polynomial with $f_{\alpha},f_{\alpha'} \in \mathcal F_J$. Assume that $\deg S^{\EK}(f_{\alpha},f_{\alpha'}) = s$.
We can decompose it as
\[
S^{\EK}(f_{\alpha},f_{\alpha'}) = x^\eta f_{\alpha} - x^{\eta'}f_{\alpha'} = \sum_{x^\gamma \in J_s} E_{\alpha\alpha'\gamma}\, x^\gamma + \sum_{x^\delta \in \cN(J)_s}
E_{\alpha\alpha'\delta}\, x^\delta
\]
where the coefficients  are equal to
\[
E_{\alpha\alpha'\gamma} \text{ (resp.~$E_{\alpha\alpha'\delta}$)} = \begin{cases}
0, & \text{if } x^\gamma \text{ (resp.~$x^\delta$)} \notin \Supp\big(S^{\EK}(f_{\alpha},f_{\alpha'})\big), \\
C_{\alpha'\beta'} - C_{\alpha\beta}, & \text{if } x^\gamma \text{ (resp.~$x^\delta$)} \in \Supp (x^\eta f_{\alpha}) \cap \Supp (x^{\eta'} f_{\alpha'}) 
\setminus \{x^\eta x^\alpha\},\\
-C_{\alpha\beta},& \text{if } x^\gamma \text{ (resp.~$x^\delta$)} \left\{\begin{array}{l} \in \Supp (x^{\eta} f_{\alpha}) \setminus \{x^\eta x^\alpha\}\\ 
\notin \Supp (x^{\eta'} f_{\alpha'}) \setminus \{x^\eta x^\alpha\}\end{array}\right.,\\
C_{\alpha'\beta'},& \text{if } x^\gamma \text{ (resp.~$x^\delta$)} \left\{\begin{array}{l} \notin \Supp (x^{\eta} f_{\alpha}) \setminus \{x^\eta x^\alpha\}\\ 
\in \Supp (x^{\eta'} f_{\alpha'}) \setminus \{x^\eta x^\alpha\}\end{array}\right..
\end{cases}
\]

The $\mathcal F_J$-reduction of $S^{\EK}(f_{\alpha},f_{\alpha'})$ is
\[
S^{\EK}(f_{\alpha},f_{\alpha'}) \xrightarrow{\mathcal F_J^{(\cdot)}}_\ast \sum_{x^\delta \in \cN(J)_s} \left( E_{\alpha\alpha'\delta} + \sum_{x^\gamma \in J_s}
E_{\alpha\alpha'\gamma} D_{\gamma\delta} \right) x^\delta.
\]

For any $\alpha,\alpha'$ such that $x^\alpha,x^{\alpha'} \in \B_J$ are involved in a syzygy of Eliahou-Kervaire type, we set
\begin{equation}\label{eq:genIdealRep}
P_{\alpha\alpha'}^\delta: = E_{\alpha\alpha'\delta} + \sum_{x^\gamma \in J_s} E_{\alpha\alpha'\gamma} D_{\gamma\delta}, \qquad s = 
\deg S^{\EK}(f_\alpha,f_{\alpha'}),\quad \forall\ x^\delta \in \cN(J)_{s}.
\end{equation}

\begin{corollary}\label{cor:mfRepresentable}
Let $J$ be a strongly stable ideal and let $\mathfrak{I}_J$  be  the ideal in $\ZZ[\CC]$ given in Theorem \ref{th:markedBasisCharacterization}. Then $\mathfrak I_J$ is  
 the ideal generated by all polynomials $P_{\alpha\alpha'}^\delta$ described in \eqref{eq:genIdealRep}. 
\end{corollary}
\begin{proof} Let $\mathfrak I'$ be the ideal generated by such  polynomials $P_{\alpha\alpha'}^\delta$. The inclusion  $\mathfrak I'\subseteq \mathfrak I_J$ follows directly from the 
construction, since the above polynomials  are $\xx$-coefficients of elements in 
$\cN(J,(\mathcal F_J) )$.

For the opposite inclusion  we consider again the  $J$-marked set  $\overline{\mathcal F}_J$  image of $\mathcal F_J$ in $\ZZ[\CC]/\mathfrak I'$. 
By construction   $\overline{\mathcal F}_J$  satisfy the condition $(\ref{it:markedBasisCharacterization_ii})$ of Theorem  \ref{th:markedBasisCharacterization}, so that it
is a $J$-marked basis. 
Therefore   $\ZZ[\CC] \rightarrow \ZZ[\CC]/\mathfrak I'$ factors as  $\ZZ[\CC] \rightarrow\ZZ[\CC]/\mathfrak I_J \rightarrow  \ZZ[\CC]/\mathfrak I'$ and
$\mathfrak I_J \subseteq   \mathfrak I'$. 
\end{proof}

To determine  equations defining  $\MFScheme{J}$ we can use Corollary \ref{cor:mfRepresentable}, namely  the criterion for marked bases in terms 
of  syzygies  given in Theorem  \ref{th:markedBasisCharacterization} that was first introduced in \cite{CioffiRoggero}   and refined in \cite{BCLR} in terms of EK-syzygies.  In particular,   Theorem  \ref{th:markedBasisCharacterization} and  Corollary \ref{cor:mfRepresentable}  
give new proofs  in terms of marked functors of  \cite[Corollary 4.6]{BCLR} and  \cite[Theorem 4.1]{CioffiRoggero}.

\begin{example}\label{ex:mainMF}
Let us compute the equations defining the scheme representing the functor $\MFFunctor{J}$ with $J = (x_2^2,x_2x_1,x_1^3) \subseteq	\ZZ[x_0,x_1,x_2]$. 
We start considering the marked set
\small
\[
\begin{split}
 f_{\text{\tiny 002}} & {} = x_2^{2}+	C_{\text{\tiny 002,020}} x_1^{2}+C_{\text{\tiny 002,101}} x_2 x_0+C_{\text{\tiny 002,110}} x_1 x_0+C_{\text{\tiny 002,200}} x_0^{2}\\
 f_{\text{\tiny 011}} &{}= x_2 x_1+C_{\text{\tiny 011,020}} x_1^{2}+C_{\text{\tiny 011,101}} x_2 x_0+C_{\text{\tiny 011,110}} x_1 x_0+C_{\text{\tiny 011,200}} x_0^{2}\\
 f_{\text{\tiny 030}} & = x_1^{3}+C_{\text{\tiny 030,120}} x_1^{2} x_0+C_{\text{\tiny 030,201}} x_2 x_0^{2}+C_{\text{\tiny 030,210}} x_1 x_0^{2}+C_{\text{\tiny 030,300}} x_0^{3}. 
\end{split}
\]
\normalsize
There are two $\textsc{EK}$-polynomials:
\small
\[
\begin{split}
 S^{\textsc{EK}} (f_{\text{\tiny 011}},f_{\text{\tiny 002}}) ={}& x_2 f_{\text{\tiny 011}} - x_1 f_{\text{\tiny 002}} = C_{\text{\tiny 011,020}} x_2 x_1^{2}-C_{\text{\tiny 
 002,020}} x_1^{3}+C_{\text{\tiny 011,101}} x_2^{2} x_0+{}\\
 &(-C_{\text{\tiny 002,101}}+C_{\text{\tiny 011,110}}) x_2 x_1 x_0 - C_{\text{\tiny 002,110}} x_1^{2} x_0+C_{\text{\tiny 011,200}} x_2 x_0^{2}-C_{\text{\tiny 002,200}}
 x_1 x_0^{2},\\
S^{\textsc{EK}} (f_{\text{\tiny 030}},f_{\text{\tiny 011}}) ={}& x_2 f_{\text{\tiny 030}} - x_1^2 f_{\text{\tiny 011}} =-C_{\text{\tiny 011,020}} x_1^{4}+
(-C_{\text{\tiny 011,101}}+C_{\text{\tiny 030,120}}) x_2 x_1^{2} x_0-C_{\text{\tiny 011,110}} x_1^{3} x_0\\ &{}+C_{\text{\tiny 030,201}} x_2^{2} 
x_0^{2}+C_{\text{\tiny 030,210}} x_2 x_1 x_0^{2}-C_{\text{\tiny 011,200}} x_1^{2} x_0^{2}+C_{\text{\tiny 030,300}} x_2 x_0^{3}.\\
\end{split}
\]
\normalsize
Since $\Supp\big(S^{\textsc{EK}} (f_{\text{\tiny 011}},f_{\text{\tiny 002}})\big) \cap J = \{x_2x_1^2,x_1^3,x_2^2x_0,x_2x_1x_0\}$ and $\Supp\big(S^{\textsc{EK}} 
(f_{\text{\tiny 030}},f_{\text{\tiny 011}})\big) \cap J = \{x_1^4,$ $x_2x_1^2x_0,x_1^3x_0,x_2^2x_0^2,x_2x_1x_0^2\}$, to perform the $\xrightarrow{\mathcal F_J^{(\cdot)}}_\ast$ 
reduction, we need some elements of $\widetilde{\mathcal F}_J^{(3)}$ and $\widetilde{\mathcal F}_J^{(4)}$. Reducing $S^{\textsc{EK}} (f_{\text{\tiny 011}},f_{\text{\tiny 002}})$ by
\small
\[
\begin{split}
\widetilde{f}_{\text{\tiny 111}} ={}& x_0 f_{\text{\tiny 011}},\qquad\qquad \widetilde{f}_{\text{\tiny 102}} = x_0 f_{\text{\tiny 002}},\qquad\qquad 
\widetilde{f}_{\text{\tiny 030}} = f_{\text{\tiny 030}},\\ 
\widetilde{f}_{\text{\tiny 021}} ={}& x_1 f_{\text{\tiny 011}} - C_{\text{\tiny 011,020}}\widetilde{f}_{\text{\tiny 030}} - C_{\text{\tiny 011,101}}
\widetilde{f}_{\text{\tiny 111}} =  x_2 x_1^{2}+(-C_{\text{\tiny 011,020}} C_{\text{\tiny 011,101}}-C_{\text{\tiny 011,020}} C_{\text{\tiny 030,120}}+
C_{\text{\tiny 011,110}}) x_1^{2} x_0+ {}\\ & (-C_{\text{\tiny 011,101}}^{2}-C_{\text{\tiny 011,020}}
      C_{\text{\tiny 030,201}}) x_2 x_0^{2}+(-C_{\text{\tiny 011,101}} C_{\text{\tiny 011,110}}-C_{\text{\tiny 011,020}} C_{\text{\tiny 030,210}}+
      C_{\text{\tiny 011,200}}) x_1 x_0^{2}+{}\\ &(-C_{\text{\tiny 011,101}}
      C_{\text{\tiny 011,200}}-C_{\text{\tiny 011,020}} C_{\text{\tiny 030,300}}) x_0^{3}\\
\end{split}
\]
\normalsize
we obtain
\footnotesize
\[
\begin{split}
 \underbrace{\left(\begin{array}{c}C_{\text{\tiny 011,020}}^{2} C_{\text{\tiny 011,101}}+C_{\text{\tiny 011,020}}^{2} C_{\text{\tiny 030,120}}
 +C_{\text{\tiny 002,101}} C_{\text{\tiny 011,020}}-C_{\text{\tiny 002,020}} C_{\text{\tiny 011,101}}\\{}-2C_{\text{\tiny 011,020}} C_{\text{\tiny 011,110}} 
 + C_{\text{\tiny 002,020}} C_{\text{\tiny 030,120}}-C_{\text{\tiny 002,110}}\end{array}\right)}_{P_{\text{\tiny 011,002}}^{\text{\tiny 120}}}& x_1^{2} x_0\\
{}+ \underbrace{\left(\begin{array}{c}C_{\text{\tiny 011,020}} C_{\text{\tiny 011,101}}^{2}+C_{\text{\tiny 011,020}}^{2} C_{\text{\tiny 030,201}}-
C_{\text{\tiny 011,101}} C_{\text{\tiny 011,110}}\\ {}+C_{\text{\tiny 002,020}} C_{\text{\tiny 030,201}}+
C_{\text{\tiny 011,200}}\end{array}\right)}_{P_{\text{\tiny 011,002}}^{\text{\tiny 201}}}& x_2 x_0^{2}\\
{}+\underbrace{\left(\begin{array}{c}C_{\text{\tiny 011,020}} C_{\text{\tiny 011,101}} C_{\text{\tiny 011,110}}
+C_{\text{\tiny 011,020}}^{2} C_{\text{\tiny 030,210}}-C_{\text{\tiny 002,110}} C_{\text{\tiny 011,101}}
+C_{\text{\tiny 002,101}} C_{\text{\tiny 011,110}}\\{}-C_{\text{\tiny 011,110}}^{2}-C_{\text{\tiny 011,020}}
C_{\text{\tiny 011,200}}+C_{\text{\tiny 002,020}} C_{\text{\tiny 030,210}}-
C_{\text{\tiny 002,200}}\end{array}\right)}_{P_{\text{\tiny 011,002}}^{\text{\tiny 210}}}& x_1 x_0^{2}\\
{}+\underbrace{\left(\begin{array}{c}C_{\text{\tiny 011,020}}C_{\text{\tiny 011,101}} C_{\text{\tiny 011,200}}+
C_{\text{\tiny 011,020}}^{2} C_{\text{\tiny 030,300}}-C_{\text{\tiny 002,200}} C_{\text{\tiny 011,101}}\\{}+
C_{\text{\tiny 002,101}} C_{\text{\tiny 011,200}}-C_{\text{\tiny 011,110}} C_{\text{\tiny 011,200}}+C_{\text{\tiny 002,020}} 
C_{\text{\tiny 030,300}}\end{array}\right)}_{{P}_{\text{\tiny 011,002}}^{\text{\tiny 300}}}& x_0^{3}.
\end{split}
\]
\normalsize
To reduce the second $\textsc{EK}$-polynomial we need
\small
\[
\begin{split}
\widetilde{f}_{\text{\tiny 211}} ={} & x_0 \widetilde{f}_{\text{\tiny 111}} = x_0^2 f_{\text{\tiny 011}},\quad \widetilde{f}_{\text{\tiny 202}} = 
x_0 \widetilde{f}_{\text{\tiny 102}} = x_0^2 f_{\text{\tiny 002}},\quad \widetilde{f}_{\text{\tiny 130}} = x_0 f_{\text{\tiny 030}},\quad 
\widetilde{f}_{\text{\tiny 121}} = x_0 \widetilde{f}_{\text{\tiny 021}},\\
\widetilde{f}_{\text{\tiny 040}} ={} & x_1f_{\text{\tiny 030}} - C_{\text{\tiny 030,120}} \widetilde{f}_{\text{\tiny 130}} - C_{\text{\tiny 030,201}}
\widetilde{f}_{\text{\tiny 211}} = x_1^{4}+(-C_{\text{\tiny 030,120}}^{2}-C_{\text{\tiny 011,020}} C_{\text{\tiny 030,201}}+C_{\text{\tiny 030,210}}) x_1^{2} x_0^{2}+\\
&(-C_{\text{\tiny 011,101}}
      C_{\text{\tiny 030,201}}-C_{\text{\tiny 030,120}} C_{\text{\tiny 030,201}}) x_2 x_0^{3}+(-C_{\text{\tiny 011,110}} C_{\text{\tiny 030,201}}-C_{\text{\tiny 030,120}} 
      C_{\text{\tiny 030,210}}+C_{\text{\tiny 030,300}}) x_1
      x_0^{3}+\\& (-C_{\text{\tiny 011,200}} C_{\text{\tiny 030,201}}-C_{\text{\tiny 030,120}} C_{\text{\tiny 030,300}}) x_0^{4}.
\end{split}
\]
\normalsize
Thus, the reduction of $S^{\textsc{EK}} (f_{\text{\tiny 030}},f_{\text{\tiny 011}})$ is
\footnotesize
\[
\begin{split}
 (\underbrace{-C_{\text{\tiny 011,020}} C_{\text{\tiny 011,101}}^{2}-C_{\text{\tiny 011,020}}^{2} C_{\text{\tiny 030,201}}+C_{\text{\tiny 011,101}} 
 C_{\text{\tiny 011,110}}-C_{\text{\tiny 002,020}}
      C_{\text{\tiny 030,201}}-C_{\text{\tiny 011,200}}}_{P_{\text{\tiny 030,011}}^{\text{\tiny 220}}})& x_1^{2} x_0^{2}\\
{}+\underbrace{\left(\begin{array}{c}-C_{\text{\tiny 011,101}}^{3}+C_{\text{\tiny 011,101}}^{2} C_{\text{\tiny 030,120}}-2 C_{\text{\tiny 011,020}}
C_{\text{\tiny 011,101}}
      C_{\text{\tiny 030,201}}-C_{\text{\tiny 002,101}} C_{\text{\tiny 030,201}}\\{}+C_{\text{\tiny 011,110}} C_{\text{\tiny 030,201}}-C_{\text{\tiny 011,101}}
      C_{\text{\tiny 030,210}}+C_{\text{\tiny 030,300}}\end{array}\right)}_{P_{\text{\tiny 030,011}}^{\text{\tiny 301}}}& x_2
      x_0^{3}\\
{}+\underbrace{\left(\begin{array}{c}-C_{\text{\tiny 011,101}}^{2} C_{\text{\tiny 011,110}}+C_{\text{\tiny 011,101}} C_{\text{\tiny 011,110}} C_{\text{\tiny 030,120}}
-C_{\text{\tiny 011,020}} C_{\text{\tiny 011,110}}
      C_{\text{\tiny 030,201}}-C_{\text{\tiny 011,020}} C_{\text{\tiny 011,101}} C_{\text{\tiny 030,210}}\\{}+C_{\text{\tiny 011,101}} C_{\text{\tiny 011,200}}
      -C_{\text{\tiny 011,200}} C_{\text{\tiny 030,120}}-C_{\text{\tiny 002,110}}
      C_{\text{\tiny 030,201}}+C_{\text{\tiny 011,020}} C_{\text{\tiny 030,300}}\end{array}\right)}_{P_{\text{\tiny 030,011}}^{\text{\tiny 310}}}& x_1 x_0^{3}\\
{}+\underbrace{\left(\begin{array}{c}-C_{\text{\tiny 011,101}}^{2} C_{\text{\tiny 011,200}}+C_{\text{\tiny 011,101}} C_{\text{\tiny 011,200}}
      C_{\text{\tiny 030,120}}-C_{\text{\tiny 011,020}} C_{\text{\tiny 011,200}} C_{\text{\tiny 030,201}}\\{}-C_{\text{\tiny 011,020}} C_{\text{\tiny 011,101}}
      C_{\text{\tiny 030,300}}-C_{\text{\tiny 002,200}}
      C_{\text{\tiny 030,201}}-C_{\text{\tiny 011,200}} C_{\text{\tiny 030,210}}+C_{\text{\tiny 011,110}} 
      C_{\text{\tiny 030,300}}\end{array}\right)}_{P_{\text{\tiny 030,011}}^{\text{\tiny 400}}}& x_0^{4}.
\end{split}
\]
\normalsize
In order to have a $J$-marked basis, the $J$-reduction of the $\textsc{EK}$-polynomials has to be 0, so that the functor $\MFFunctor{J}$ is represented by the scheme
\[
\MFScheme{J} = \Spec \ZZ[\CC]/\left(P_{\text{\tiny 011,002}}^{\text{\tiny 120}},P_{\text{\tiny 011,002}}^{\text{\tiny 201}},P_{\text{\tiny 011,002}}^{\text{\tiny 210}}
,P_{\text{\tiny 011,002}}^{\text{\tiny 300}},P_{\text{\tiny 030,011}}^{\text{\tiny 220}},P_{\text{\tiny 030,011}}^{\text{\tiny 301}},P_{\text{\tiny 030,011}}^{\text{\tiny 310}},
P_{\text{\tiny 030,011}}^{\text{\tiny 400}}\right).
\]

Now, for any ring $A$,  each  element of $\MFFunctor{J}(A)$  is given by a  scheme morphism $\Spec A \rightarrow \MFScheme{J}$, or equivalently by a ring morphism $\ZZ[\CC] \rightarrow A$ that factors through $\ZZ[\CC] \rightarrow \ZZ[\CC]/\mathfrak I_J \rightarrow A$. For instance, for $A=\ZZ[t]$, the ring morphism $\ZZ[\CC] \rightarrow \ZZ[t]$ 
given by
\small
\[
\begin{array}{lllllll}
C_{\text{\tiny 002,020}} \mapsto 1-t && C_{\text{\tiny 002,101}} \mapsto 0&& C_{\text{\tiny 002,110}} \mapsto t^3-t^4&& C_{\text{\tiny 002,200}} \mapsto -t^2 \\
C_{\text{\tiny 011,020}} \mapsto 0&& C_{\text{\tiny 011,101}} \mapsto 0&& C_{\text{\tiny 011,110}} \mapsto t && C_{\text{\tiny 011,200}} \mapsto t^2-t\\
C_{\text{\tiny 030,120}} \mapsto t^3 && C_{\text{\tiny 030,201}} \mapsto  t && C_{\text{\tiny 030,210}} \mapsto 0 && C_{\text{\tiny 030,300}} \mapsto -t^2 \\
\end{array}
\]
\normalsize
factors through $\ZZ[\CC] \rightarrow \ZZ[\CC]/\mathfrak I_J  \rightarrow \ZZ[t]$, therefore the following is a  $J$-marked basis in $\ZZ[t][x_0,x_1,x_2]$
\small
\[\begin{array}{l}
f_{\text{\tiny 002}}=x_2^{2}+ (1-t) x_1^{2} - (t^4-t^3) x_1 x_0 - t^2\, x_0^{2}, \\ 
 f_{\text{\tiny 011}}=x_2 x_1+t\, x_1 x_0+(t^2-t) x_0^{2},\\
 f_{\text{\tiny 030}}=x_1^{3}+ t^3\, x_1^{2} x_0+ t\, x_2 x_0^{2} -t^2\, x_0^{3}. 
\end{array}
\]
\normalsize
\end{example}

\section{Marked schemes and truncation ideals}

An ideal $I \in \MFFunctor{J}(A)$ defines a quotient algebra $A[\xx]/I$ that is a free $A$-module, so that the family $\Proj A[\xx]/I \rightarrow \Spec A$ 
is flat and defines a morphism from $\Spec A$ to a suitable Hilbert scheme, by the universal property of Hilbert schemes. Thus, it is natural to study the 
relation between marked schemes and Hilbert schemes. Since Hilbert schemes parametrize flat families of subschemes of a projective space, and the same subscheme can
be defined by infinitely many different ideals, we first need to investigate the function that associates to every  ideal in $\MFFunctor{J}(A)$  the
scheme in  $\PP^n_A$ it  defines. In general, this function  can  be non-injective,  as the following example shows.

\begin{example}[{cf.~\cite[Example 3.4]{BCLR}}]
Consider the strongly stable ideal $J = (x_2,x_1^2,x_1x_0)$ in $\ZZ[x_0,x_1,x_2]$. For any ring $A$ and 
$a \in A$, consider the $J$-marked set $F_{J,a} = \{x_2 + a \, x_1,x_1^2,x_1x_0\}$. These marked sets are in fact 
 $J$-marked bases, since the unique $\textsc{EK}$-polynomial involving the first generator
\[
S^{\textsc{EK}}(x_2 + a \, x_1, x_1^2) = x_1^2 (x_2 + a \, x_1) - x_2 (x_1^2) = a x_1^3
\] 
is clearly contained in $\langle F_{J,a}^{(3)}\rangle$. Moreover, for every $a$, the  ideal $(F_{J,a})_{\geqslant 2}$ coincides with $J_{\geq 2} \otimes A$,
since $x_2^2 = x_2 (x_2 + a \, x_1) - a x_1 (x_2 + a \, x_1) + a^2 (x_1^2)$, $x_2x_1 = x_1 (x_2 + a \, x_1) - a (x_1^2)$ and $x_2x_0 = x_0 (x_2 + a \, x_1) - a (x_1x_0)$. 
Therefore, for all $a \in A$, the ideals $(F_{J,a})$ define the same scheme $\Proj A[\xx]/J$.
\end{example}


The following proposition states that non-uniqueness is a consequence of divisibility by $x_0$.
 
\begin{proposition}[{cf.~\cite[Theorem 3.3]{BCLR}}]\label{diverse}
Let $J$ be a strongly stable ideal and let $m$  be  the  minimum  degree such that $J_m\neq 0$. Assume that no monomial of degree larger than  $m$ in the monomial basis $\B_J$  
is divisible by   $x_0$ (or equivalently that  $x_0^t \cN(J)_{\geq m}\subseteq \cN(J)_{\geq m+t}$ for every $t$). Then for any   
two different $J$-marked bases  $F_J$ and $G_J$  in $A[\xx]$,  the schemes $\Proj A[\xx]/(F_J)$ and $\Proj A[\xx]/(G_J)$ are different. 
\end{proposition}
\begin{proof}
By hypothesis and by Proposition \ref{prop:markedSetExists}\emph{(\ref{it:markedSetExists_ii})}, there exists a monomial $x^\alpha \in \mathcal{B}_J$ such that 
the corresponding polynomials $f_\alpha \in F_J$ and $g_\alpha \in G_J$ are different. If $\Proj A[\xx]/(F_J) = \Proj A[\xx]/(G_J)$, then $(F_J)_{\geq s} = (G_J)_{\geq s}$
for a sufficiently large $s$. Therefore, for $s \gg 0$, $x_0^{s}  f_\alpha$ is contained in $(G_J)$ and $x_0^s f_\alpha - x_0^s g_\alpha = x_0^{s} \big(\T(g_\alpha) -
\T(f_\alpha)\big) \in (G_J)$. By definition, the support of $\T(g_\alpha) - \T(f_\alpha)$ is contained in $\cN(J)$.  Therefore,  also the support
of $x_0^s  \big(\T(g_\alpha) - \T(f_\alpha)\big)$ is in $\cN(J)$,   due to the hypothesis on $J$.  Finally, by Theorem \ref{th:markedSetChar}\emph{(\ref{it:markedSetChar_ii})}-\emph{(\ref{it:markedSetChar_vi})}, 
this implies  $x_0^s \big(\T(g_\alpha) - \T(f_\alpha)\big) \in (F_J)_{s+\vert \alpha \vert}  \cap \langle \cN(J) \rangle=\{0\}$, so that 
$\T(g_\alpha) =\T(f_
\alpha)$, against the assumption $f_\alpha \neq g_\alpha$.
\end{proof}


\begin{definition}
We say that $J$ is a \emph{$m$-truncation ideal} (\textit{$m$-truncation} for short)  if $J= J'_{\geq m}$ for $J'$  a saturated strongly stable ideal.
\end{definition}

 Observe that the  monomials divisible by $x_0$ in the monomial basis of a $m$-truncation ideal $J$ (if any) are of degree $m$. Therefore, by Proposition \ref{diverse}, different $m$-truncation ideals define different projective schemes. We emphasize that a priori the truncation degree $m$ can be any positive integer. We will discuss special values of $m$ later in the paper. 

We now describe the relations among  marked functors (resp.~schemes)  corresponding to different truncations of the same saturated strongly stable ideal $J$.
We will prove that for sufficiently large integers $m$, the $J_{\geq m}$-marked schemes are all isomorphic. However, 
the construction of $\MFScheme{J_{\geq m}}$ given in Theorem \ref{th:mfRepresentable} depends on $m$ since we obtain it as a closed 
subscheme of an affine
space whose dimension increases with $m$. From a computational point of view it will be convenient to choose, among isomorphic marked schemes, the one corresponding to the minimum value of $m$, while for other applications higher  values of $m$ can be  more convenient.
 
In order to compare $J_{\geq m}$-marked bases in $A[\xx]$  for different values of $m$, we refer to Proposition \ref{diverse}. By associating to a marked 
basis the scheme it defines, we will identify $I=(F_{J_{\geq m}})\in  \MFFunctor{J_{\geq m}}(A)$ and  $I'=(G_{J_{\geq m'}})\in  \MFFunctor{J_{\geq m'}}(A)$ when 
$\Proj A[\xx]/I=\Proj A[\xx]/I'$ in $\PP^n_A$, i.e.~when  $I_{\geq s}=I'_{\geq s}$, for $s\gg 0$. By Theorem \ref{th:markedSetChar}\emph{(\ref{it:markedSetChar_iii})} and Proposition \ref{prop:markedSetExists}\emph{(\ref{it:markedSetExists_ii})}, this is equivalent to $\widetilde{F}_{J_{\geq m}}^{(s)}=\widetilde{G}_{J_{\geq m'}}^{(s)}$ for $s\gg 0$.

\begin{theorem}\label{th:markedFunctorsInclusion}
Let $J$ be a saturated strongly stable ideal. 
Then, for every  $s > 0$ and for any noetherian ring  $A$, $\MFFunctor{J_{\geq s-1}}(A)
\subseteq \MFFunctor{J_{\geq s}}(A)$. More precisely,
\begin{enumerate}[(i)]
\item\label{it:markedFunctorsInclusion_i} if $J$ has no minimal generators of degree $s+1$ divisible by the variable $x_1$ or $J_{\geqslant s-1} = J_{\geqslant s}$,
then $\MFFunctor{J_{\geq s-1}} = \MFFunctor{J_{\geq s}}$;
\item\label{it:markedFunctorsInclusion_ii} otherwise, $\MFFunctor{J_{\geq s-1}}$ is a proper closed subfunctor of $\MFFunctor{J_{\geq s}}$.
\end{enumerate}
\end{theorem}
\begin{proof}
To prove the inclusion $\MFFunctor{J_{\geq s-1}}(A) \subseteq \MFFunctor{J_{\geq s}}(A)$, let us consider a $J_{\geqslant s-1}$-marked basis $F$.
The set  $G := \widetilde{F}^{(s)}  \cup \{ f_{\alpha} \in F \ \vert\ x^\alpha \in \B_J \text{ and } \vert\alpha\vert >  s \}$ is by construction a
$J_{\geqslant s}$-marked set. In fact, $G$ is a $J_{\geq s}$-marked basis, since $\langle G^{(s)}\rangle = \langle F^{(s)}\rangle$ by 
Theorem \ref{th:markedSetChar}\emph{(\ref{it:markedSetChar_iii})-(\ref{it:markedSetChar_iv})} and the generators of degree larger than $s$ are the same in the two marked sets.

  From now on in this proof we denote by $J'$  the truncation of $J$ in degree $s-1$, by $\mathcal{F}_{J'}$ the marked set analogous to the one given in \eqref{eq:MarcataParametri} that we use to construct the ideal  $\mathfrak{I}_{J'} \subseteq \ZZ[\CC']$ of $\MFScheme{J'}$. We also let  
$A':=\ZZ[\CC']/\mathfrak{I}_{J'}$,  $\phi_{F_{J'}}\colon \ZZ[\CC']\rightarrow A'$ the canonical map on the quotient and $\phi_{F_{J'}}[\xx]$ the extension to $\ZZ[\CC'][\xx]\rightarrow A'[\xx]$.
 Moreover,  $J''$ will be the truncation of $J$ in degree $s$
and $\mathcal{F}_{J''}$, $\ZZ[\CC'']$, $\mathfrak{I}_{J''}$, $A''$, $\phi_{F_{J''}}$ are defined analogously. By the definition of $\mathfrak{I}_{J'}$ and $\mathfrak{I}_{J''}$, we observe that 
$\phi_{F_{J'}}[\xx](\mathcal{F}_{J'})$ is a $J'$-marked
basis in $A'[\xx]$ and $\phi_{F_{J''}}[\xx](\mathcal{F}_{J''})$ is a $J''$-marked
basis in $A''[\xx]$.

We first prove \emph{(\ref{it:markedFunctorsInclusion_ii})}. Let us consider the $J''$-marked set 
\[
\mathcal{G}:= \widetilde{\mathcal{F}}_{J'}^{(s)} \cup \left\{ f'_{\alpha} \in \mathcal{F}_{J'} \ \vert\ x^\alpha \in \B_J ,\  \vert\alpha\vert >  s\right\}.
\]
  By Theorem \ref{th:markedSetChar}\emph{(\ref{it:markedSetChar_v})-(\ref{it:markedSetChar_viii})}, $\phi_{F_{J'}}[\xx](\mathcal{G})$ is a $J''$-marked basis of 
$A'[\xx]$, since 
\[
\cN\left(J'', \big(\phi_{F_{J'}}[\xx](\mathcal G)\big)\right)\subseteq \cN\left(J', \big(\phi_{F_{J'}}[\xx](\mathcal{F}_{J'})\big)\right)=\{0\}.
\]
Thus, 
the ring homomorphism 
\[
\begin{split}
\psi \colon \ZZ[\CC''] &\longrightarrow \ZZ[\CC'] \\
\parbox{1cm}{\centering $C''_{\alpha \beta}$} &\longmapsto \text{coefficient of } x^\beta \text{ in } g_\alpha \in {\mathcal G}.
\end{split}
\]
induces a homomorphism $\overline{\psi} \colon A'' \rightarrow A'$ such that $\phi_{F_{J'}}\circ \psi=\overline{\psi} \circ \phi_{F_{J''}}$.  Moreover $\phi_{F_{J'}}\circ \psi$  
is surjective, being the composition of two surjective homomorphisms. Indeed,
\[
 C'_{\alpha\beta} = \begin{cases}
\psi(C''_{\alpha\beta}),& \text{if }  x^\alpha \in \B_J ,\  \vert\alpha\vert \geq  s, \\
\psi(C''_{\eta \gamma}),\ x^{\eta} =x_0 x^\alpha ,\   x^{\gamma} = x_0x^\beta ,& \text{otherwise}
\end{cases}.
\]
Under our assumptions, for every $f'_\alpha \in \mathcal{F}_{J'}$ of degree $s-1$, $x_0T(f'_\alpha)$ is a $J''$-remainder, so that $x_0f'_\alpha \in \mathcal{G}$.   

Therefore, the epimorphism $\overline{\psi}$ induces an  isomorphism between $\MFScheme{J'} = \Spec \ZZ[\CC']/\mathfrak{I}_{J'}$ and a closed subscheme of  $\MFScheme{J''} = \Spec \ZZ[\CC'']/\mathfrak{I}_{J''}$. In order to show that this subscheme is proper, we can look at the Zariski tangent spaces of $\MFScheme{J'}$ and $\MFScheme{J''}$ at the points corresponding to $J'$ and $J''$ and see that they have different dimension (see \cite[Theorem 5.7]{BCLR} for the details).
%

\smallskip

To prove \emph{(\ref{it:markedFunctorsInclusion_i})}  we observe that the new condition  on $J$  implies that for every $x^\gamma \in \cN(J)_s$, either $x_1x^\gamma \in \cN(J)_{s+1}$  or $x_1x^\gamma=x_0x^\delta $ with $x^\delta \in J_s$ holds. 

Exploiting this property we first prove that $C''_{\eta\gamma}\in \mathfrak{I}_{J''}$ if $x^\eta \in J_s$, $x_0\mid x^{\eta}$ and $x_0 \nmid x^{\gamma}$.  Let 
$x^{\epsilon}=x_1x^\eta/x_0$ and consider the EK-syzygy  $S^{\EK}(f''_{\eta} , f''_{\epsilon})=x_0T(f''_\epsilon) - x_1T( f''_{\eta})$ between the elements $ f''_{\eta}, f''_\epsilon \in \mathcal{F}_{J''}$.  The $J$-remainder of this polynomial given by 
$\xrightarrow{\mathcal{F}_{J''}^{(\cdot)}}$ is of the type  $g=x_0T(f''_\epsilon)-x_1T( f''_{\eta})+x_0 \sum C_{\eta \delta}'' f''_{\beta}$,
where  $f''_{\beta}\in \mathcal{F}_{J''}$  and the sum is over the multi-indices $\beta$ such that $x^{\beta}:=x_1 x^\delta /x_0 \in J_s$ with $x^{\delta}$ divisible by $x_0$ and contained in the support of $T( f''_{\eta})$ and $C_{\eta \delta}''$ the coefficient of $x^{\delta}$ in $f''_{\eta}$. If  $ x^{\gamma} $ is a term in the support of $  T( f''_{\eta})$ such that $x_0\nmid x^{\gamma}$, then  $x_1 x^{\gamma} \in \cN(J)_{s+1}$ is contained in the support of $ g$. By definition, 
$\mathfrak{I}_{J''}$ contains the $\xx$-coefficients of $g$, thus in particular the coefficient $C''_{\eta \gamma}$ of $x_1 x^{\gamma}$ in $g$.  For every $x^\alpha \in J_{s-1}$ and  $x^\eta=x_0x^\alpha$, let us denote by $h_\alpha$ the polynomial in $\ZZ[\CC''][\xx]$ such that 
$f''_\eta=x_0 h_\alpha+\sum C''_{\eta \gamma}x^\gamma$  with $x_0\nmid x^\gamma$, so that $\phi_{F_{J''}}[\xx](f''_\eta)=\phi_{F_{J''}}[\xx](x_0 h_\alpha)$.

Using these polynomials we can define the   $J'$-marked set 
\[
\mathcal H =\{h_\alpha   \ \vert \ x^\alpha\in J_{s-1} \}\cup \{  f''_\eta  \in \mathcal{F}_{J''}  \ \vert \  x^\eta  \in \B_J,\ \vert \eta \vert \geq s \}.
\]
 By construction, $\phi_{F_{J''}}[\xx](x_0 \mathcal H)\subseteq  \phi_{F_{J''}}[\xx]( \mathcal{F}_{J''})$, therefore $\phi_{F_{J''}}[\xx](\mathcal H)$ is a $J'$-marked basis    by Theorem \ref{th:markedSetChar} 
\emph{(\ref{it:markedSetChar_v})-(\ref{it:markedSetChar_viii})}. In fact, if the support of an element   $u$ in the ideal ${}_{A''} \big(\phi_{F_{J''}}[\xx](\mathcal{H})\big)$ only contains monomials of $\cN(J)$, then $x_0 u$ has the same support and is in ${}_{A''} \big(\phi_{F_{J''}}[\xx]( \mathcal{F}_{J''})\big)$, so that  
$u=0$ since $ \cN\big(J,\big(\phi_{F_{J''}}[\xx](\mathcal{F}_{J''})\big)\big)=\{0\}$.

Thus, the ring homomorphism 
\[
\begin{split}
\varphi \colon \ZZ[\CC'] &\longrightarrow \ZZ[\CC''] \\
C'_{\alpha \beta} &\longmapsto \text{coefficient of } x^\beta \text{ in }  h_\alpha  \text{ if  } \vert \alpha\vert =s-1  \\
 C'_{\eta \gamma} &\longmapsto \text{coefficient of } x^\gamma  \text{ in }  f''_\eta  \text{ if  } x^\eta \in \B_J, \ \vert \eta \vert \geq s
\end{split}
\]
induces a homomorphism $\overline{\varphi} \colon A' \rightarrow A''$.

Finally, $\overline{\psi}$ and $\overline{\varphi}$ are inverses of each other. 
Indeed, if we apply  to the $J'$-marked set  $\mathcal{H}$ the  construction from the first part of the proof, we obtain a  $J''$-marked set $\mathcal G'$ such that $\phi_{F_{J''}}[\xx](\mathcal G')$ is a $J''$-marked basis and $\phi_{F_{J''}}[\xx](\mathcal{G}')\subseteq \phi_{F_{J''}}[\xx](\mathcal{F}_{J''})$, hence $\phi_{F_{J''}}[\xx](\mathcal{G}')= \phi_{F_{J''}}[\xx](\mathcal{F}_{J''})$ by Proposition \ref{prop:markedSetExists}.
\end{proof}


\section{Marked schemes and Hilbert schemes}

We now briefly recall how the Hilbert scheme can be constructed as subscheme of a suitable Grassmannian. 
For any positive integer $n$ and any numerical polynomial $p(t)$, consider the Hilbert functor
\begin{equation*}
\HilbFunctor{n}{p(t)}:\ \underline{\text{Noeth-Rings}}\ \longrightarrow\ \Sets
\end{equation*}
associating to any noetherian ring $A$ the set
\begin{equation*}
\HilbFunctor{n}{p(t)}(A) = \left\{ X \subseteq \PP^n_A\ \vert\ X \rightarrow \Spec A \text{ is flat and has fibers with Hilbert polymial } p(t)\right\}
\end{equation*}
and to any ring homomorphism  $f: A \rightarrow B$ the map
\[
\begin{split}
\HilbFunctor{n}{p(t)}(f):\  \HilbFunctor{n}{p(t)}(A)\ &\longrightarrow\ \parbox{3cm}{\centering $\HilbFunctor{n}{p(t)}(B)$}\\
 \parbox{2cm}{\centering $X$} &\longmapsto\ X \times_{\Spec A} \Spec B.
\end{split}
\]
Grothendieck first defined this functor and showed that it is representable \cite{Groth}. The Hilbert scheme $\HilbScheme{n}{p(t)}$ is defined as the scheme representing the Hilbert
functor and it is classically constructed as a subscheme of a suitable Grassmannian. Let us briefly recall how (for a detailed exposition see 
\cite{IarrobinoKleiman,HaimSturm,BLMR}).
By Gotzmann's Regularity theorem (\cite[Satz (2.9)]{Gotzmann} and \cite[Lemma C.23]{IarrobinoKleiman}), there exists a positive integer $r$ only depending on $p(t)$,
called \emph{Gotzmann number}, for which the ideal sheaf $\mathcal{I}_X$ of each scheme $X \in \HilbFunctor{n}{p(t)}(A)$ is $r$-regular (in the sense of 
Castelnuovo-Mumford regularity). This implies that the morphism
\[
H^0\big(\OO_{\PP^n_A}(r)\big)\ \xrightarrow{\phi_X}\ H^0\big(\OO_X(r)\big)
\]
is surjective. By flatness, $H^0(\OO_X(r))$ is a locally free module of rank $p(r)$ and, as an $A$-module, $H^0(\OO_{\PP^n_A}(r))$ is isomorphic to the homogeneous 
piece of degree $r$ of the polynomial ring $A[\xx]$. Since $A[\xx]_r \simeq A^N$, where $N = \binom{n+r}{n}$, the homomorphism $\phi_X$ may be viewed as an element of the Grassmannian, whose corresponding functor is
\begin{equation*}
\GrassFunctor{p(r)}{N}:\ \underline{\text{Noeth-Rings}}\ \longrightarrow\ \Sets
\end{equation*}
associating to any noetherian ring $A$ the set
\begin{equation*}
\GrassFunctor{p(r)}{N}(A) = \left\{\begin{array}{c}\text{isomorphism classes of epimorphisms } A^N \rightarrow Q\\ \text{of locally free modules of rank }
p(r)\end{array} \right\}
\end{equation*}
and to any morphism $f: A \rightarrow B$ the map
\[
\begin{split}
\GrassFunctor{p(r)}{N}(f):\  \GrassFunctor{p(r)}{N}(A)\ &\longrightarrow\ \parbox{3cm}{\centering $\GrassFunctor{p(r)}{N}(B)$}\\
 \parbox{2cm}{\centering $A^N \rightarrow Q$} &\longmapsto\ B^N \rightarrow Q \otimes_A B.
\end{split}
\]
Two epimorphisms $\phi: A^N \rightarrow Q$ and $\phi': A^N \rightarrow Q'$ are isomorphic if there exists an isomorphism $\psi:Q \rightarrow Q'$ of $A$-modules such that the diagram
\[
\begin{tikzpicture}
\node (1) at (0,0) [] {$A^N$};
\node (2) at (0,-1.25) [] {$A^N$};
\node (3) at (2,0) [] {$Q$};
\node (4) at (2,-1.25) [] {$Q'$};

\draw [->] (1) --node[above]{\tiny $\phi$} (3);
\draw [->] (2) --node[above]{\tiny $\phi'$} (4);
\draw [->] (1) --node[left]{\tiny $\text{id}$} (2);
\draw [->] (3) --node[right]{\tiny $\psi$} (4);
\end{tikzpicture}
\]
commutes. Equivalently, $\phi$ and $\phi'$ are isomorphic if $\ker \phi = \ker \phi'$. Therefore, by identifying isomorphism classes of epimorphisms $\phi$ with $\ker \phi$, the Grassmann functor sends $A$ to the set
\[
\left\{\begin{array}{c}A\text{-submodules } M \subseteq A^N \text{ such that}\\ A^N/M \text{ is locally free of rank } p(r)\end{array} \right\}.
\]

This functor is representable and the representing scheme $\GrassScheme{p(r)}{N}$ is the Grassmannian (see \cite[Section 16.7]{Vakil}).
Therefore, one of the possible embeddings of the Hilbert scheme into a Grassmannian is given by the natural transformation of functors (introduced by Bayer in \cite{Bayer82})
\begin{equation}\label{eq:BayerEmbedding}
\underline{\mathscr{H}}:\  \HilbFunctor{n}{p(t)}\ \longrightarrow\ \GrassFunctor{p(r)}{N}
\end{equation}
given by
\[
\begin{split}
  \HilbFunctor{n}{p(t)}(A)\ &\quad\longrightarrow\quad \parbox{3cm}{\centering $\GrassFunctor{p(r)}{N}(A)$}\\
 \parbox{2cm}{\centering $X$} &\quad\longmapsto\quad A[\xx]_r \twoheadrightarrow H^0\big(\OO_X(r)\big).
\end{split}
\]
By Yoneda's Lemma, any natural transformation of representable functors is induced by a 
unique morphism between their  representing schemes. The associated morphism $\mathscr{H}:\HilbScheme{n}{p(t)} \rightarrow \GrassScheme{p(r)}{N}$ is a closed embedding 
and the equations defining the Hilbert scheme $\HilbScheme{n}{p(t)}$ as a subscheme of $\GrassScheme{p(r)}{N}$ were conjectured by Bayer \cite{Bayer82} and proved much 
later by Haiman and Sturmfels \cite{HaimSturm}.

The Grassmannian has the well-known open cover by affine spaces which also defines the Pl\"ucker embedding. For any set $\mathcal{N}$ of $p(r)$ distinct monomials of
$A[\xx]_r$, consider the map
\[
\i_{\mathcal{N}}: {}_A\langle \mathcal{N}\rangle \simeq A^{p(r)} \hookrightarrow A[\xx]_r \simeq A^N
\]  
and the subfunctor $\GFunctor{\mathcal{N}}$ such that
\begin{equation*}
\GFunctor{\mathcal{N}}(A) = \left\{ \text{classes } \phi_Q: A^N \rightarrow Q \text{ in } \GrassFunctor{p(r)}{N}(A) \text{ such that } \phi_Q \circ \i_{\mathcal{N}} 
\text{ is surjective}   \right\}.
\end{equation*}
Each such subfunctor is open and the family obtained varying the set of monomials $\mathcal{N}$ covers the Grassmann functor \cite[Lemma 22.22.1]{stacks-project}.
Since $\phi_Q \circ \i_{\mathcal{N}}$ is an epimorphism between a free module and a locally free module of the same rank, it is in fact an isomorphism. Therefore, each 
$Q$ in $\GFunctor{\mathcal{N}}(A)$ can be identified with the free module ${}_A\langle \mathcal{N}\rangle$ and we can  rewrite the functors $\GFunctor{\mathcal{N}}$ as
\begin{equation*}
\GFunctor{\mathcal{N}}(A) = \left\{\text{epimorphisms } A[\xx]_r \rightarrow {}_A\langle\mathcal{N}\rangle \text{ of free modules of
rank } p(r) \right\}
\end{equation*}
An epimorphism $\phi:A[\xx]_r \rightarrow {}_A\langle \mathcal{N}\rangle$
is determined by its values on basis elements, 
\[
\phi(x^\alpha) = \sum_{x^\beta \in \mathcal{N}} a_{\alpha,\beta}x^\beta.
\]
 Thus its kernel is generated by 
\[
f_\alpha := x^\alpha - \sum_{x^\beta \in \mathcal{N}} a_{\alpha \beta}\, x^\beta
\]
for all $x^\alpha$ of total degree $r$  lying outside $\mathcal{N}$.
 If $J$ is the ideal generated by the monomials in $A[\xx]_r$ not contained in  $\mathcal{N}$, then  we can describe $\GFunctor{\mathcal{N}}$ as 
\begin{eqnarray*}
\notag \GFunctor{\mathcal{N}}(A) & = & \left\{ \text{free submodules } L \subseteq A[\xx]_r \text{ such that } A[\xx]_r \simeq L \oplus {}_A\langle\mathcal{N}\rangle \right\}
= {} \\
&=&  \left\{ \text{submodules } L \subseteq A[\xx]_r \text{ generated by a $J$-marked set} \right\}.
\end{eqnarray*}

We are interested in the open subfunctors of the Hilbert functor
$\HilbFunctor{n}{p(t)}$ induced by the family of subfunctors $\GFunctor{\mathcal{N}}$ by means of $\underline{\mathscr{H}}$. We denote by $\HFunctor{\mathcal{N}}$  the subfunctor associating to $A$ the set
\begin{equation}\label{eq:openH}
\HFunctor{\mathcal{N}}(A) := \underline{\mathscr{H}}^{-1}\left( \GFunctor{\mathcal{N}}(A)\cap \underline{\mathscr{H}}\left(\HilbFunctor{n}{p(t)}(A)\right) \right).
\end{equation}

The kernel of the map $A[\xx]_r \rightarrow H^0(\OO_X(r))$ is represented by the global sections of the sheaf $\mathcal{I}_X(r)$, i.e.~by the homogeneous piece of 
degree $r$ of the saturated ideal $I_X$ defining $X$. Since $I_X$ and $(I_X)_{\geqslant r}$ define the same scheme and $(I_X)_{\geqslant r}$ is generated by the
homogeneous piece of degree $r$, we can rewrite the subfunctor $\HFunctor{\mathcal{N}}$ as
\begin{eqnarray}
\notag\HFunctor{\mathcal{N}}(A) & = & \left\{  X \in \HilbFunctor{n}{p(t)}(A)\ \big\vert\ A[\xx]_r \simeq H^0 \big(\mathcal{I}_X(r)\big) 
\oplus {}_A\langle\mathcal{N}\rangle \right\}  \\
\label{eq:openHquotient}
& = & \left\{  X \in \HilbFunctor{n}{p(t)}(A)\ \big\vert\ (I_X)_{\geqslant r} \text{ is generated by a $J$-marked set}     \right\}.
\end{eqnarray}

It is then natural to relate $\HFunctor{\mathcal{N}}(A)$ to $\MFFunctor{J}(A)$. In general their relations are less obvious than one might expect. However, 
under suitable conditions on $\cN$ and  $J$,  we can identify $\HFunctor{\mathcal{N}}$ with a marked functor. The following result gives a new proof in terms 
of functors of \cite[Theorem 2.5]{BLR}.

\begin{lemma}
 Let $p(t)$ be a Hilbert polynomial in $\PP^n$ with Gotzmann number $r$  and  let $J$ be a strongly stable ideal such that  $\vert \cN(J)_r\vert=p(r)$.  Then, for every noetherian ring $A$
\[
\HFunctor{{\cN(J)_r}}(A)\neq \emptyset\quad \Longleftrightarrow\quad  \text{the Hilbert polynomial of } A[\xx]/J \text{ is }p(t).
\]
\end{lemma}
\begin{proof}
$(\Leftarrow)$ If the Hilbert polynomial of $A[\xx]/J $ is $p(t)$, then  $\Proj A[\xx]/J \in \HFunctor{\mathcal{N}}(A)$.

$(\Rightarrow)$ Assume that $X$ is a scheme in  $\HFunctor{\mathcal{N}}(A)$ and set $I:=(I_X)_{\geqslant r}$. 
By Theorem \ref{th:markedSetChar}\emph{(\ref{it:markedSetChar_iii})}, for  every $m\geq r$, the $A$-module $I_{m}$ 
has a free direct summand  with rank equal to that of $J_m$, therefore the value of the Hilbert polynomial of $J$ in every degree $m\geq r$ cannot be smaller than $p(m)$. 
On the other hand, this rank cannot be larger than $p(m)$ by Macaulay's Estimate on the Growth of Ideals \cite[Theorem 3.3]{Green}.
%
%
\end{proof}

\begin{corollary}\label{cor:isomorf}
Let  $J$ be a saturated  strongly stable ideal  such that $\ZZ[\xx]/J$ has  Hilbert polynomial $p(t)$. 
Then 
\[
 \HFunctor{\cN(J)_r}\simeq \MFFunctor{J_{\geq r}}.
\]
\end{corollary}
We can rephrase the statement of the corollary by saying that for every noetherian ring $A$,
\begin{equation*}\label{eq:openHideal}
\HFunctor{\mathcal{N}(J)_r}(A)= \left\{  X \in \HilbFunctor{n}{p(t)}(A)\ \big\vert\  (I_X)_{\geqslant r} 
\text{ is generated by a $J$-marked basis} \right\}.
\end{equation*} 
Therefore, upon identifying ideals and the schemes they define, the isomorphism from the corollary is a canonical identification $\HFunctor{\mathcal{N}(J)_r} = \MFFunctor{J_{\geq r}}$.

We can then deduce from Corollary \ref{cor:isomorf} an isomorphism between the representing schemes.  Taking into account Theorem \ref{th:markedFunctorsInclusion}
we obtain the following result. 

\begin{corollary}\label{cor:mfEmbHilb}
Let $J$ be a saturated strongly stable ideal and $r$ be the Gotzmann number of its  Hilbert
polynomial $p(t)$. If $\rho$ is the maximum degree of monomials in $\B_J$  divisible by $x_1$, then
\begin{enumerate}[(i)]
\item\label{it:mfEmbHilb_i} for $s \geqslant \rho-1$,  $\MFScheme{J_{\geq s}}$ is an open subscheme of $\HilbScheme{n}{p(t)}$;
\item\label{it:mfEmbHilb_ii} for $s < \rho-1$,  $\MFScheme{J_{\geq s}}$ is a locally closed subscheme of $\HilbScheme{n}{p(t)}$.
\end{enumerate}
\end{corollary}
\begin{proof} 
\emph{(\ref{it:mfEmbHilb_i})} By Theorem \ref{th:markedFunctorsInclusion}, we have
\[
\MFScheme{J_{\geq \rho - 1}} \simeq \MFScheme{J_{\geq \rho}} \simeq \cdots \simeq \MFScheme{J_{\geq r}} = \HScheme{\cN(J)_r}.
\]

\emph{(\ref{it:mfEmbHilb_ii})} By Theorem \ref{th:markedFunctorsInclusion}, for $s < \rho-1$, we know that in the chain
\[
\MFScheme{J_{\geq s}} \subseteq \MFScheme{J_{\geq s+1}} \subseteq \cdots \subseteq \MFScheme{J_{\geq \rho - 1}} \simeq \cdots \simeq \MFScheme{J_{\geq r}} = \HScheme{\cN(J)_r},
\]
there is at least one proper closed embedding, so that $\MFScheme{J_{\geq s}}$ is a locally closed subscheme of the Hilbert scheme.
\end{proof}

\begin{remark} \label{rk:meglio}
Though our results only apply to  a small amount of the open subsets $ \HFunctor{\mathcal{N}} $ that are necessary to cover
$\HilbScheme{n}{p(t)}$,
in many interesting cases they allow to obtain  a different open cover  by exploiting  the  action of the linear group $\textnormal{PGL}(n+1)$.
 In particular, this is true for the Hilbert scheme $\HilbScheme{n,K}{p(t)}=\HilbScheme{n}{p(t)} \times_{\Spec \ZZ} \Spec K$ representing  the Hilbert functor $\HilbFunctor{n,K}{p(t)}\colon \underline{K\text{-Algebras}} \rightarrow 
\Sets$ for every field $K$ of characteristic zero. Indeed, the properties of the generic initial ideal 
proved by Galligo \cite{Galligo} allow to prove that every point of the Hilbert scheme  is contained in an open subset $\HScheme{\mathcal{N},K}$, where $\mathcal{N} := \cN(J)_r$
for a saturated strongly stable ideal $J$, at least up to the action of a general element in $\textnormal{PGL}(n+1)$.
 Such open cover of $\HilbScheme{n,K}{p(t)}$ is presented in \cite{BLR,BLMR,BBR}. 

The set of strongly stable ideals that are necessary to obtain such new open cover of the Hilbert scheme 
can be effectively determined using the algorithm presented in \cite{CLMR,LellaBorel,LellaWeb}. 
\end{remark}

\begin{remark}\label{rk:IK}
The equations of the open subscheme $\mathbf{H}_{\cN(J)_r}$ computed as the marked scheme over $J_{\geq r}$ are the same equations determined by Iarrobino and Kleiman in \cite{IarrobinoKleiman}. Indeed, the Eliahou-Kervaire syzygies among the generators of $J_{\geq r}$ are linear, so that imposing $S^{\EK}(f_\alpha,f_\beta) \xrightarrow{F_{J\geq r}^{(\cdot)}} 0$ is equivalent to prove that $\langle\textit{SF}_{J_{\geq r}}^{(r+1)}\rangle \subseteq \langle F_{J_{\geq r}}^{(r+1)}\rangle$. If we represent the generators $\{x_i f_\alpha\ \vert\ \forall\ x^\alpha \in \mathcal{B}_{J_{\geq r}},\ i = 0,\ldots,n\}$ of $(F_J)_{r+1}$ by a matrix $\mathcal{M}_J^{(r+1)}$, then the condition $\langle\textit{SF}_{J_{\geq r}}^{(r+1)}\rangle \subseteq \langle F_{J_{\geq r}}^{(r+1)}\rangle$ is equivalent to $\rank \mathcal{M}_J^{(r+1)} \leqslant \rank \langle F_{J_{\geq r}}^{(r+1)}\rangle = \rank J_{r+1} = \binom{n+r}{n}-p(r)$ and the latter condition is guaranteed by imposing the vanishing of the minors of order $\rank J_{r+1}+1$. This is how Iarrobino-Kleiman determine local equations of the Hilbert scheme. Notice that using this approach it is possible to deduce that the equations are of degree at most $\binom{n+r}{n}-p(r)+1$, while constructing the equations applying Theorem \ref{th:markedBasisCharacterization}\emph{(\ref{it:markedBasisCharacterization_iii})} and our reduction procedure, it is possible to deduce that the equations have degree at most $\deg p(t) + 2$ (see \cite[Theorem 3.3]{BLR}).
\end{remark}

\begin{remark}
The statements of  Corollary \ref{cor:mfEmbHilb}  can be very useful both from a computational and a theoretical  point of view. Indeed,  for a fixed saturated ideal $J$, 
the number of variables involved 
in  the computation of
equations defining the marked
scheme $\MFScheme{J_{\geq s}}$ dramatically  increases with $s$. On the other hand, in \cite{BBR} the equalities of Corollary \ref{cor:mfEmbHilb}\emph{(\ref{it:mfEmbHilb_i})}
show that the open subset of $\HilbScheme{n,K}{p(t)}$ of the 
$r'$-regular points, for a given   $r'<r$, can be embedded as a locally closed subscheme in the  Grassmannian $\GrassScheme{p(r')}{N(r')}$, smaller than that in which we can embed
the entire Hilbert scheme.

Moreover, in several cases  marked schemes $\MFScheme{J_{\geq s}} $ with $s<\rho -1$
correspond to interesting loci of the Hilbert scheme and our results allow effective computations also on them.
\end{remark}

\begin{example}
Consider the strongly stable ideal $J = (x_2^2,x_2x_1,x_1^4) \subseteq \ZZ[x_0,x_1,x_2]$. The Hilbert polynomial of $\Proj \ZZ[x_0,x_1,x_2]/J$ is $p(t) = 5$ with Gotzmann number equal to $5$. Therefore, the open subscheme $\mathbf{H}_{\cN(J)_5} \subseteq \HilbScheme{2}{5}$ can be defined as closed subscheme of the affine open subscheme $\mathbf{G}_{\cN(J)_5} \subseteq \GrassScheme{5}{21}$ of dimension $80$. Applying Corollary \ref{cor:mfEmbHilb}\emph{(\ref{it:mfEmbHilb_i})}, we can define the same open subscheme by means of the isomorphism $\MFScheme{J_{\geq 3}} \simeq \mathbf{H}_{\cN(J)_5}$ with $\MFScheme{J_{\geq 3}} \subseteq \AA^{30}$.

Finally, also the marked scheme associated to the saturated ideal may be very important. For instance, in the special case of zero-dimensional schemes in the projective plane $\PP^2$, for each postulation there is a unique strongly stable ideal $J$ realizing it (see for instance \cite[Chapters 1-3]{EisenbudSyzygies}). Therefore, $\MFScheme{J}$ parametrizes the locus of the Hilbert scheme $\HilbScheme{2}{d}$ with a fixed Hilbert function (up to the action of the projective linear group). In the example, the scheme $\MFScheme{J}$ parameterize the locus of 5 points in the plane with postulation $(1,3,4,5,\ldots)$.
\end{example}

\section{Gr\"obner strata}


Throughout this section, we denote by $\sigma$ a term ordering on the polynomial ring $A[\xx]$ and by $\IN{\sigma}(I)$ the initial ideal of an ideal $I \subseteq A[\xx]$ w.r.t.~such term ordering.  We define the \emph{Gr\"obner functor} $\GSFunctor{J}{\sigma}: \underline{\text{Noeth-Rings}} \rightarrow 
\Sets$ that associates to any ring $A$ the set
\begin{equation}\label{eq:gsSet}
\GSFunctor{J}{\sigma}(A) = \big\{ \text{ideals } I\subseteq A[\xx]\ \vert\ \IN{\sigma}(I) = J \big\}
\end{equation}
and to any ring homomorphism $\phi: A \rightarrow B$ the function
\[
\begin{split}
\GSFunctor{J}{\sigma}(\phi):\ \GSFunctor{J}{\sigma}(A)\ &\longrightarrow\ \GSFunctor{J}{\sigma}(B)\\
\parbox{1.5cm}{\centering $I$}\ & \longmapsto\ I \otimes_A B.
\end{split}
\]

Gr\"obner basis theory over rings is more intricate than Gr\"obner basis theory over fields (see also \cite{Lederer} for a more detailed discussion). A first delicate issue is the definition of initial ideals. Given an ideal $I \subseteq A[\xx]$, we can consider the 
ideal generated by the leading monomials  or the ideal generated by leading terms, i.e.~monomials with coefficients, of the polynomials in $I$. 
In general neither of the two  definitions is well-suited for functorial constructions, since taking the initial ideal of a given $I \subseteq A[\xx]$ does not commute with base change $\otimes_A B$ unless the initial ideal of $I$ is a monomial ideal. For instance,  
the initial ideal of $I=(2x_1+x_0) \subseteq \ZZ[x_0,x_1]$, $x_1>x_0$, is  $J'=(x_1)$ according to the first definition and 
$J''=(2x_1)$ according to the second one; after the extension $\ZZ\rightarrow \ZZ_2:=\ZZ/2\ZZ$ we obtain $\In_\sigma(I\otimes_\ZZ \ZZ_2)=(x_0)$, while $J'\otimes_\ZZ \ZZ_2= (x_1)$ and $J''\otimes_\ZZ \ZZ_2= (0)$.


\begin{definition}[\cite{Pauer,Wibmer}]
Let $I$ be an ideal in a polynomial ring $A[\xx]$, with $A$ a noetherian ring, and let $\sigma$ be a term ordering. The ideal $I$ is called \emph{monic} 
(w.r.t.~$\sigma$) if for all monomials $x^\alpha \in A[\xx]$ the set
\[
\textnormal{LC}(I,x^\alpha) = \left\{ a \in A\ \vert\ ax^\alpha \text{ is the leading term of } g \in I   \right\} \cup \{0\}
\]
is either $\{0\}$ or $A$.
\end{definition}

Therefore, the definition  of $\GSFunctor{J}{\sigma}$  given in \eqref{eq:gsSet} is correct and non-ambiguous if we assume that $J$ is a monomial ideal and 
restrict the set of ideals $I$ to those that are monic. To this aim, we follow the line of the definition of marked functor and  consider  the ideals $I$ 
that are generated by a  suitable set of polynomials, marked on $J$, that we expect to form a reduced Gr\"obner basis. Indeed, an ideal $I \subseteq A[\xx]$ admits 
a reduced Gr\"obner basis if, and only if, $I$ is a monic ideal (see \cite{Aschenbrenner,Pauer,Wibmer}). We recall that a  \emph{reduced} Gr\"obner basis is a Gr\"obner basis composed by polynomials with leading coefficient equal to $1_A$ and such that no term other than the leading one is contained in the initial ideal.
More precisely, the reduced Gr\"obner basis takes the shape
\[
G_J = \left\{ x^\alpha - \sum_{\mathclap{\cramped{\begin{subarray}{c} x^\beta \in \cN(J)_{\vert\alpha\vert} \\ x^\alpha >_\sigma x^\beta\end{subarray}}}} b_{\alpha\beta}\, x^\beta \in A[\xx]\ 
\Bigg\vert \ x^\alpha \in \B_J\right\}.
\]
This is a $J$-marked set, considering the marking given by the term ordering, i.e.~$\Ht(g) = \IN{\sigma}(g)$. Furthermore, $G_J$ is a $J$-marked basis, since 
for 
 $I=(G_J)\in \GSFunctor{J}{\sigma}(A)$, the  monomials in
 $\cN(\IN{\sigma}(I))=\cN(J)$ are even a basis of the $A$-module $A[\xx]/I$.
Then we can rewrite $\GSFunctor{J}{\sigma}(A)$ as
\begin{align}
\notag \GSFunctor{J}{\sigma}(A) &{} =  \big\{ \text{monic ideals } I\subseteq A[\xx]\ \vert\ \IN{\sigma}(I) = J \big\} = {}\\
\label{eq:gsBases} & {}= \left\{ I = \left(G_J\right)\ \vert\ G_J \text{ reduced Gr\"obner basis and } \IN{\sigma}\big((G_J)\big) = J\right\}.
\end{align}
Thus,  $\GSFunctor{J}{\sigma}(A) \subseteq \MFFunctor{J}(A)$ for every $A$, and there is an injection of functors $\GSFunctor{J}{\sigma} \rightarrow \MFFunctor{J}$.


\begin{lemma}
Let $J$ be any monomial ideal and $\sigma$ be a term ordering. Then $\GSFunctor{J}{\sigma}$ is a functor and a Zariski sheaf.
\end{lemma}
\begin{proof}
The arguments used in the proofs of Proposition \ref{prop:mfFunctor} and Lemma \ref{lem:mfZariskiSheaf} also apply to the case of the Gr\"obner functor.
\end{proof}

\begin{theorem}\label{th:gsClosedSubfunctor}
Let $J$ be a $m$-truncation strongly stable ideal and $\sigma$ be a term ordering.
Then the Gr\"obner functor $\GSFunctor{J}{\sigma}$ is a closed subfunctor of $\MFFunctor{J}$. 

Using Notation \ref{notazioni},  $\GSFunctor{J}{\sigma}$ is represented by the affine
scheme $\GSScheme{J}{\sigma}:=\Spec \ZZ[\CC]/\mathfrak{I}_J^{\sigma}$, where $\mathfrak{I}_J^{\sigma}$ is the sum of the ideal $\mathfrak{I}_J$ described 
in Theorem \ref{th:mfRepresentable} 
and the ideal $ \mathfrak{G}^\sigma_J := ( C_{\alpha\beta}\ \vert\ x^\beta >_\sigma x^\alpha)$.
\end{theorem}
\begin{proof}
Straightforward by applying the criterion given in Proposition 2.9 of \cite{HaimSturm} on the inclusion $\iota: \GSFunctor{J}{\sigma}(A) \hookrightarrow \MFFunctor{J}(A)$.
\end{proof}

 We will call \emph{Gr\"obner stratum} the scheme representing the Gr\"obner functor. 

\begin{example} 
Let us consider the ideal $J = (x_2^2,x_2x_1,x_1^3)$ of Example \ref{ex:mainMF} and the term ordering $\mathtt{DegLex}$. There is only one monomial in $\cN(J)$  greater
than some monomial of the same degree in $\B_J$: $x_2x_0^2 >_{\mathtt{DegLex}} x_1^3$. Therefore, the ideal defining 
$\GSScheme{J}{\mathtt{DegLex}}$ as a subscheme of $\AA^{12}=\Spec \ZZ[\CC]$ is the sum of the ideal defining $\MFScheme{J}$ and the principal ideal  $(C_{\text{\tiny 030,201}})$ and $\GSScheme{J}{\mathtt{DegLex}}$ is a hyperplane section of $\MFScheme{J}$.
\end{example}

An analogue of Theorem \ref{th:markedFunctorsInclusion} also holds for Gr\"obner strata (see \cite[Theorem 4.7]{LR}). In particular, we have an isomorphism 
$\GSScheme{J_{\geq s-1}}{\sigma} \simeq \GSScheme{J_{\geq s}}{\sigma}$ under the assumption of
Theorem \ref{th:markedFunctorsInclusion}\emph{(\ref{it:markedFunctorsInclusion_i})}, leading to the isomorphism $\MFScheme{J_{\geq s-1}} \simeq \MFScheme{J_{\geq s}}$.
From this property, we can deduce some cases in which marked families and Gr\"obner strata coincide.

We need the following property.
\begin{proposition}[{\cite[Lemma 3.2]{CLMR}}]
Let $J$ be a saturated strongly stable ideal. If the truncation $J_{\geq m}$ is a gen-segment ideal, then so is $J_{\geq m-1}$. In general,
the opposite implication is not true.
\end{proposition}

Thus, if we consider a strongly stable saturated ideal $J$ without minimal generators divisible by $x_1$ in degree $s+1$, then  $\MFScheme{J_{\geq s-1}}
\simeq \MFScheme{J_{\geq s}}$ and $\GSScheme{J_{\geq s-1}}{\sigma} \simeq \GSScheme{J_{\geq s}}{\sigma}$. If, moreover, we assume that there exists a term 
ordering $\sigma$ making $J_{\geq s-1}$ a gen-segment ideal, then by Theorem \ref{th:markedFunctorsInclusion}, we get
\[
\GSScheme{J_{\geq s}}{\sigma} \simeq \GSScheme{J_{\geq s-1}}{\sigma} = \MFScheme{J_{\geq s-1}} \simeq \MFScheme{J_{\geq s}}
\]
so that $\GSScheme{J_{\geq s}}{\sigma}$ and $\MFScheme{J_{\geq s}}$ coincide,  even if $J_{\geq s}$ were not a gen-segment ideal. Note that in this last case there 
exist pairs of monomials $x^\alpha \in J_s$  and $ x^\beta \in \cN(J)s$ such that $ x^\alpha <_{\sigma} x^\beta$. However, since $\mathfrak{I}_{J_{\geq s}}$ and 
$\mathfrak{I}_{J_{\geq s}}^{\sigma}$ coincide, the variables  $C_{\alpha\beta} $ corresponding to those pairs of monomials must be  already
contained in $\mathfrak{I}_{J_{\geq s}}$.

\begin{example}
Let us consider the ideal $J = (x_2^3,x_2^2x_1,x_2x_1^2) \subseteq	\ZZ[x_0,x_1,x_2]$. Its  Castelnuovo-Mumford regularity is $3$ and its Hilbert polynomial 
is   $p(t) = t+4$ with Gotzmann number  $4$.

For $s =1,2,3$, $J_{\geq s}$ is a gen-segment ideal w.r.t.~any term ordering $\sigma$ induced by a refinement of the grading $(4,3,1)$, 
whereas $J_{\geq 4}$ cannot be a gen-segment since $x_2x_1^2x_0 \in J_4$, $x_1^4,x_2^2x_0^2 \in \cN(J)_4$ and $(x_2x_1^2x_0)^2 = x_1^4 \cdot x_2^2x_0^2$.
Since there is no minimal generator in degree $5$, the equality $\MFScheme{J_{\geq 3}} = \GSScheme{J_{\geq 3}}{\sigma}$ induces the equality $\MFScheme{J_{\geq 4}} =
\GSScheme{J_{\geq 4}}{\sigma}$ as subschemes of $\HilbScheme{2}{t+4}$, even if our construction defines them in affine spaces of different dimensions. 
Indeed, in the construction of $\MFScheme{J_{\geq 4}}$ we consider the variable $C_{\text{\tiny 121,040}}$ corresponding to the monomial $x_1^4$ 
in the tail of the polynomial $f_{\text{\tiny 121}}$ with $\Ht(f_{\text{\tiny 121}})=x_2x_1^2x_0 $, while this variable does not appear 
in the construction of  $\GSScheme{J_{\geq 4}}{\sigma}$, 
since  $x_1^4 >_\sigma x_2x_1^2x_0$. This means that the variable $C_{\text{\tiny 121,040}}$ must be  already contained 
in the ideal defining $\MFScheme{J_{\geq 4}}$. We will now check this fact, by a direct computation of   $\mathfrak{I}_{J_{\geq 4}}$
as done in Corollary \ref{cor:mfRepresentable}.

Among the $\text{EK}$-polynomials involving $f_{\text{\tiny 121}}$ there is
$g:=S^{\EK}(f_{\text{\tiny 121}},f_{\text{\tiny 130}}) = x_1f_{\text{\tiny 121}} - x_0f_{\text{\tiny 130}}$.
The only monomials in $\Supp(g) \cap J $ are $ x_2^2x_1x_0^2$ and $x_2x_1^2x_0^2$, both divisible by $x_0$.
Then
\[
g\xrightarrow{\mathcal{F}_{J_{\geq 4}}^{(\cdot)}}_\ast h=g-(C_{\text{\tiny 121,202}}-C_{\text{\tiny 130,211}})\widetilde{f}_{212}-(C_{\text{\tiny 121,112}}-C_{\text{\tiny 130,121}})\widetilde{f}_{212}
\]
where
 $\widetilde{f}_{212} = x_0 f_{211}$ and $\widetilde{f}_{122} = x_0 f_{121}$. Therefore, we replace the monomials $x_2^2x_1x_0^2$ and $x_2x_1^2x_0^2$ with  
linear combinations of monomials all divisible by $x_0$, so that the monomial $x_1^5$ still appears in the support of $h$ with coefficient $C_{\text{\tiny 121,040}}$.
Therefore,   $C_{\text{\tiny 121,040}}$ is one of the  generators of the ideal $\mathfrak{I}_{J_{\geq 4}}$ defining $\MFScheme{J_{\geq 4}}$.
\end{example}

\section{\texorpdfstring{Example: marked schemes and Gr\"obner strata of $(x_3,x_2^2,x_2x_1^3,x_1^4)$}{Example: marked schemes and Gr\"obner strata of (x3,x2\^{}2,x2 x1\^{}3,x1\^{}4)}}\label{sec:finalExample}
In the final section, we report some results about marked schemes and Gr\"obner strata associated to the strongly stable ideal $J = (x_3,x_2^2,x_2x_1^3,x_1^4) \subseteq k[x_0,x_1,x_2,x_3]$ and its truncations. The ideal $J$ defines a point of the Hilbert scheme $\HilbScheme{3}{7}$, which is an irreducible scheme of dimension 21 \cite{Mazzola}. As the Gotzmann number is $7$, $\HilbScheme{3}{7}$ can be defined as subscheme of the Grassmannian $\GrassScheme{7}{120}$. The Iarrobino-Kleiman equations of the open subscheme $\mathbf{H}_{\cN(J)_7} \subseteq \mathbf{G}_{\cN(J)_7}$ can be computed considering the marked scheme $\MFScheme{J_{\geq 7}}$. By direct computation, one can check that $\MFScheme{J_{\geq 7}} \simeq \mathbf{H}_{\cN(J)_7}$ is defined by $2058$ quadratic equations in the coordinate ring of the affine space $\mathbb{A}^{791} \simeq \mathbf{G}_{\cN(J)_7}$. This embedding is clearly inconvenient because of the huge number of variables and the resulting large codimension of $\MFScheme{J_{\geq 7}}$. 

By Theorem \ref{th:markedFunctorsInclusion}, the marked scheme $\MFScheme{J_{\geq 7}}$ is isomorphic to the marked scheme $\MFScheme{J_{\geq 3}}$. The latter scheme is defined as subscheme of $\AA^{105}$, its ideal is generated by $210$ quadratic polynomials and it turns out to be isomorphic to a rational hypersurface $V$ in the affine space $\AA^{22}$ defined by a degree $6$ polynomial. The hardest part of the computation is working out the equations in order to find explicitly the embedding $\MFScheme{J_{\geq 3}} \hookrightarrow \AA^{22}$, since the process of elimination of 83 parameters highly increases the degree of the polynomials. This step can last a few hours (depending on the CPU) and requires large RAM memory. We recall that, in order to overcome this difficulty, an alternative polynomial reduction procedure (the so-called \emph{superminimal reduction}) has been developed in \cite{BCLR}. This procedure allows to embedded the marked scheme in an affine space of far lower dimension. For instance, $\MFScheme{J_{\geq 3}}$ can be embedded in $\AA^{28}$. Considering this embedding, we would need to eliminate only 6 parameters (instead of 83). 

The superminimal reduction procedure can be seen as a generalization of the procedure used for computing Gr\"obner strata of zero-dimensional ideals in the affine framework. However, we emphasize that the open subscheme $\mathbf{H}_{\cN(J)_7}$ cannot be studied as a Gr\"obner stratum. First, the truncation $J_{\geq 3}$ is not a gen-segment ideal. Indeed, the polynomials in the marked basis with head term $x_2^2x_0$ and $x_1^4$ have respectively $x_2x_1^2$ and $x_2x_1^2 x_0$ in the support of their tails, but $x_2x_1^2$ and $x_2x_1^2 x_0$ cannot appear at the same time in generators with initial terms $x_2^2x_0$ and $x_1^4$ of a reduced Gr\"obner basis, since we would have $x_2^2x_0 >_\sigma x_2x_1^2$, $x_1^4 >_\sigma x_2x_1^2 x_0$ and $x_2^2x_0 \cdot x_1^4 = x_2x_1^2 \cdot x_2x_1^2 x_0$. Second, the coefficients of $x_2x_1^2$ and $x_2x_1^2 x_0$ in the polynomials of the marked basis with head terms $x_2^2x_0$ and $x_1^4$ are not contained in the ideal defining $\MFScheme{J_{\geq 3}}$ so that the Gr\"obner stratum $\GSScheme{J_{\geq 3}}{\sigma}$ is a proper subscheme for every $\sigma$ (Theorem \ref{th:gsClosedSubfunctor}). The smallest codimension of a Gr\"obner stratum contained in $\MFScheme{J_{\geq 3}}$ is 1. In fact, the Gr\"obner strata corresponding to the term orderings obtained as refinement of the gradings $(13,6,4,1)$ and $(11,6,3,1)$ are both isomorphic to $\AA^{20}$. In the generic Gr\"obner basis of $\GSScheme{J_{\geq 3}}{(13,6,4,1)}$, the monomial $x_2x_1^2$ does not appear in the generator with initial term $x_2^2x_0$, while in the case of $\GSScheme{J_{\geq 3}}{(11,6,3,1)}$, $x_2x_1^2x_0$ does not appear in the generator with initial term $x_1^4$ (and there are no other differences with the marked basis).

Other proper subschemes of $\MFScheme{J_{\geq 3}}$ can be obtained considering marked schemes (and Gr\"obner strata) of truncation of $J$ in degree lower than 3 (the computation in these cases is much simpler and lasts few seconds). In Table \ref{tab:finalEx}, we show the comparison between marked schemes and Gr\"obner strata w.r.t.~the graded lexicographic and reverse lexicographic term orderings of several truncation of $J$. Notice that we already know by theoretical reasons that $\MFScheme{J_{\geq 2}} \simeq \MFScheme{J}$ and $\GSScheme{J_{\geq 2}}{\sigma} \simeq \GSScheme{J}{\sigma}$, for any $\sigma$. The case of the reverse lexicographic order is even more special, since the Gr\"obner strata w.r.t.~$\mathtt{RevLex}$ of all truncations are isomorphic \cite[Proposition 4.11]{LR}.

\begin{table}[!ht]
\begin{center}
\begin{tikzpicture}[scale=1]
\draw [thick] (3,1) -- (15,1);
\draw [thick] (1.5,0) -- (15,0);
\draw [thick] (1.5,-6) -- (15,-6);

\draw [] (1.5,-2) -- (15,-2);
\draw [] (1.5,-4) -- (15,-4);

\draw [thick] (1.5,0) -- (1.5,-6);
\draw [thick] (3,1) -- (3,-6);
\draw [] (7,1) -- (7,-6);
\draw [] (11,1) -- (11,-6);
\draw [thick] (15,1) -- (15,-6);

\node at (5,0.5) {Marked scheme};
\node at (9,0.7) {Gr\"obner stratum};
\node at (9,0.3) {w.r.t.~$\mathtt{RevLex}$};
\node at (13,0.7) {Gr\"obner stratum};
\node at (13,0.3) {w.r.t.~$\mathtt{DegLex}$};

\node at (2.25,-1) {$J$};
\node at (2.25,-3) {$J_{\geqslant 2}$};
\node at (2.25,-5) {$J_{\geqslant 3}$};

\node at (5,-0.5) {$\MFScheme{J} \subseteq \mathbb{A}^{22}$};
\node at (5,-1) {$28$ equations};
\node at (5,-1.5) {$\MFScheme{J} \simeq \mathbb{A}^{15}$};

\node at (5,-2.5) {$\MFScheme{J_{\geq 2}} \subseteq \mathbb{A}^{39}$};
\node at (5,-3) {$77$ equations};
\node at (5,-3.5) {$\MFScheme{J_{\geq 2}} \simeq \mathbb{A}^{15}$};

\node at (5,-4.3) {$\MFScheme{J_{\geq 3}} \subseteq \mathbb{A}^{105}$};
\node at (5,-4.75) {$210$ equations};
\node at (5,-5.2) {$\MFScheme{J_{\geq 3}} \simeq V \subseteq \mathbb{A}^{22}$};
\node at (5,-5.65) {$V \stackrel{\sim}{\dashrightarrow} \AA^{21}$, $\deg V = 6$};

\node at (9,-0.5) {$\GSScheme{J}{\mathtt{RevLex}} \subseteq \mathbb{A}^{22}$};
\node at (9,-1) {$28$ equations};
\node at (9,-1.5) {$\GSScheme{J}{\mathtt{RevLex}} \simeq \mathbb{A}^{15}$};

\node at (9,-2.5) {$\GSScheme{J_{\geq 2}}{\mathtt{RevLex}} \subseteq \mathbb{A}^{37}$};
\node at (9,-3) {$71$ equations};
\node at (9,-3.5) {$\GSScheme{J_{\geq 2}}{\mathtt{RevLex}} \simeq \mathbb{A}^{15}$};

\node at (9,-4.5) {$\GSScheme{J_{\geq 3}}{\mathtt{RevLex}} \subseteq \mathbb{A}^{93}$};
\node at (9,-5) {$204$ equations};
\node at (9,-5.5) {$\GSScheme{J_{\geq 3}}{\mathtt{RevLex}} \simeq \mathbb{A}^{15}$};

\node at (13,-0.5) {$\GSScheme{J}{\mathtt{DegLex}} \subseteq \mathbb{A}^{19}$};
\node at (13,-1) {$28$ equations};
\node at (13,-1.5) {$\GSScheme{J}{\mathtt{DegLex}} \simeq \mathbb{A}^{12}$};

\node at (13,-2.5) {$\GSScheme{J_{\geq 2}}{\mathtt{DegLex}} \subseteq \mathbb{A}^{36}$};
\node at (13,-3) {$77$ equations};
\node at (13,-3.5) {$\GSScheme{J_{\geq 2}}{\mathtt{DegLex}} \simeq \mathbb{A}^{12}$};

\node at (13,-4.5) {$\GSScheme{J_{\geq 3}}{\mathtt{DegLex}} \subseteq \mathbb{A}^{102}$};
\node at (13,-5) {$210$ equations};
\node at (13,-5.5) {$\GSScheme{J_{\geq 3}}{\mathtt{DegLex}} \simeq \mathbb{A}^{18}$};
\end{tikzpicture}
\end{center}
\caption{\label{tab:finalEx} Marked schemes and Gr\"obner strata w.r.t.~the graded lexicographic and reverse lexicographic term orderings of $J$, $J_{\geq 2}$ and $J_{\geq 3}$, where $J = (x_3,x_2^2,x_2x_1^3,x_1^4) \subseteq k[x_0,x_1,x_2,x_3]$.}
\end{table}

Moreover, by direct computation, we observe that the saturated ideal $J$ is a gen-segment w.r.t.~the reverse lexicographic order, so that $\MFScheme{J} = \GSScheme{J}{\mathtt{RevLex}}$. The truncation $J_{\geq 2}$ is a gen-segment w.r.t.~any term ordering induced by a refinement of the grading $(7,4,3,1)$, so that the marked scheme $\MFScheme{J_{\geq2}}$ coincides with the Gr\"obner stratum $\GSScheme{J_{\geq 2}}{(7,4,3,1)}$. Notice that the term ordering induced by $(7,4,3,1)$ allows two monomials more than the reverse lexicographic order in the tails of the marked set, in fact $\GSScheme{J_{\geq 2}}{(7,4,3,1)} \subseteq \mathbb{A}^{39}$ and $\GSScheme{J_{\geq 2}}{\mathtt{RevLex}} \subseteq \mathbb{A}^{37}$. However, the explicit computation shows that these two monomials cannot appear in the tails of a marked basis, since $\GSScheme{J_{\geq 2}}{(7,4,3,1)} = \MFScheme{J_{\geq2}} \simeq \MFScheme{J} = \GSScheme{J}{\mathtt{RevLex}} \simeq \GSScheme{J_{\geq 2}}{\mathtt{RevLex}}$.

\bigskip

\paragraph{\bf Acknowledgment} The authors would like to thank Mathias Lederer for his help in strongly improving the first version of this paper.

\providecommand{\bysame}{\leavevmode\hbox to3em{\hrulefill}\thinspace}
\providecommand{\href}[2]{#2}

\end{document}